\documentclass[a4paper,11pt,oneside]{article}
	\usepackage{amsfonts}
	\usepackage{float}
	\usepackage{amssymb}
	\usepackage{indentfirst}
	\usepackage{mathrsfs}
	\usepackage{amsfonts,amsmath,latexsym,bm}
	\usepackage{amsthm}
	\usepackage{booktabs}
	\usepackage{graphicx}
	\newcommand{\tabcaption}{\def\@captype{table}\caption}

\newtheorem{theorem}{Theorem}[section]
\newtheorem{lemma}[theorem]{Lemma}

\newtheorem{example}[theorem]{Example}

\newtheorem{remark}[theorem]{Remark}

\newcommand{\dxi}{\,{\rm d}\xi}
\newcommand{\deta}{\,{\rm d}\eta}

\newcommand{\bs}{\boldsymbol}

\begin{document}

\title%[Superconvergence of Hybrid stress finite element method]
{Superconvergence   and recovery type  a posteriori error estimation for hybrid stress finite element method\thanks{This work was supported by National Natural Science Foundation of China (11171239) and  Major Research Plan of  National Natural Science Foundation of China (91430105).}}

\author{Yanhong Bai\thanks{School of Mathematics, Sichuan University, Chengdu 610064, China. Email: baiyanhong1982@126.com}, \quad Yongke Wu\thanks{School of Mathematical Sciences, University of Electronic Science and Technology of China, Chengdu 611731, China. Email: wuyongke1982@uestc.edu.cn}, \quad Xiaoping Xie\thanks{Corresponding author. School of Mathematics, Sichuan University, Chengdu 610064, China. Email: xpxie@scu.edu.cn}}

\date{}
\maketitle
\begin{abstract}
Superconvergence   and a posteriori error estimators of recovery type   are analyzed  for the 4-node  hybrid stress quadrilateral finite element method proposed by Pian and Sumihara (Int. J. Numer. Meth. Engrg., 1984, 20: 1685-1695)
 for linear elasticity problems.   Uniform superconvergence of order $O(h^{1+\min\{\alpha,1\}})$ with respect to  the  Lam\'{e} constant $\lambda$ is established for both  the recovered gradients of the displacement  vector and the stress tensor under a mesh assumption,  where    $\alpha>0$ is a parameter characterizing the  distortion of meshes from parallelograms to quadrilaterals. A posteriori error
estimators based on the recovered quantities are shown to be asymptotically exact. Numerical experiments confirm the theoretical results.

\end{abstract}
\maketitle

\section{Introduction}\label{sec:Intro}

Assumed stress hybrid finite element method (also called hybrid stress   method)  pioneered by Pian \cite{Pian1964}    is known to be an efficient  approach
in the analysis of   elasticity problems   (cf. \cite{Pian1978,Pian1995,P-S,Pian-Tong1969,Pian-Tong1986,P-W,  Xie-Zhou1,Zhou-Nie,Zhou-Xie}). %, \cite{P-S}, \cite{P-W}, \cite{Zhang_Xie2010}, \cite{Zhou-Nie}, \cite{Zhou-Xie}.
One main advantage of the hybrid method lies in that, the method allows for  piecewise-independent  approximation  to the  stress   solution and,  through local elimination of the stress unknowns,   finally leads to a  symmetric and positive definite discrete system of unknowns of  displacements.
 In \cite{P-S} Pian and Sumihara derived a robust  4-node hybrid stress quadrilateral element (abbr. PS) through a rational choice of stress terms, where the continuous piecewise isoparametric bilinear  interpolation is used for the displacement approximation.  We refer to \cite{Yu-Xie-Carsten} for the analysis of uniform convergence and a posteriori error estimation for the hybrid stress quadrilateral elements proposed in \cite{P-S, Xie-Zhou1}.

As an active research topic, superconvergence   of finite element solutions to partial differential equations  has been studied intensively for conforming, nonconforming and mixed finite element methods (see, e.g.,  books \cite{Basbuska-Strouboulis,Chen2, Chen-Huang, Li-Huang-Yan2012, Lin-Yan, Wahlbin, Yan, Zhu-Lin} and papers \cite{Bank-Xu2003, Bank-Xu2003a,Chen.L2006, Ewing-Liu-Wang, Heimsund-Tai-Wang2002,Huang-Xu2008,Lakhany;2000, Li-Zhang1999, Schatz;1996,Shi-Jiang-Xue1997,Wang-Ye2001,Ye.X2002, ZM-Z3,Zhang-Naga, Zlamal1977}).  Based on theory of  superconvergence, a posteriori error estimation of recovery type has attracted more and more research interests in recent two decades. The most representative
recovery type error estimator is the Zienkiewicz-Zhu (ZZ) estimator  based on gradient patch recovery  by local discrete
least-squares fitting \cite{Z-Zhu1, Z-Zhu2}. The method is   widely used in engineering practice for its robustness.   Superconvergence properties of the ZZ patch recovery were shown
in \cite{Zhang2000,Li-Zhang1999} for  rectangular  and strongly regular triangular meshes, respectively.  The work of  \cite{Bank-Xu2003,Bank-Xu2003a}  introduced a recovery type error estimator
based on global $L^2$-projection  with smoothing iteration of the multigrid
method, and  established asymptotic exactness in the $H^1$-norm
for linear element    under shape regular triangulation. %In \cite{}
By using the
result in  \cite{Bank-Xu2003},  a new theoretical justification was given in \cite{Huang-Xu2008} for the ZZ  estimator.
A polynomial preserving gradient recovery  (PPR)  method was  proposed in \cite{ZM-Z3,Zhang-Naga}     which is different from the ZZ gradient patch recovery  method \cite{Z-Zhu1}. In \cite{Shi-Xu-Zhang} some patch recovery methods were proposed and analyzed for finite element approximation of elasticity problems using quadrilateral meshes.
%In \cite{Ming-Shi-Xu}, the author propose a superconvergence property at the element
%center, the vertices and the midpoints of four edges for the nonconforming rotated Q1 element over a mildly distorted
%quadrilateral mesh.

So far,   to the authors' knowledge,  there is no  superconvergence  analysis for  the hybrid stress  finite element method  for the elasticity problems.  This paper is to establish superconvergence  for the Pian and Sumihara's hybrid stress quadrilateral element   \cite{P-S}.    We shall derive the uniform superconvergence with respect to  the  Lam\'{e} constant $\lambda$ for both  the recovered  displacement gradients and the recovered stress tensor, and   show that  the a posteriori error
estimators  based on the recovered quantities  are asymptotically exact.

The rest of the paper is organized as follows. Section 2    introduces  the model problem and its weak form. Section 3 shows  the hybrid stress finite element  discretization and some preliminary results. %geometric properties of quadrilateral meshes.
Section 4 analyzes the  superconvergence of the hybrid stress method.  Section 5  is devoted to the recovery of the displacement gradients and the stress tensor, as well as the a posteriori estimation of recovered type.
	  Finally, Section 6 provides  numerical results.

\section{Model problem}

%In this section, we introduce the model problem firstly, then give the geometric properties of quadrilaterals meshes and finally introduce the PS \cite{P-S}% and ECQ4 \cite{Xie-Zhou1}
%hybrid stress finite element method .

%\section{Model problem}
 Let $\Omega\subset \mathbb{R}^2$ be a bounded polygonal domain with boundary $\partial \Omega$.
We consider the following linear elasticity problem with homogeneous displacement boundary condition:
\begin{equation}\label{eq:model-prob}
\left\{\begin{array}{rlll}
-\text{div}\mathbf\sigma  & = & \mathbf{f} &  \text{in}\ \Omega,\\
\mathbf\sigma  & = &   \mathbb{C}\mathbf\epsilon(\mathbf{u}) & \text{in}\ \Omega,\\
\mathbf u& = & \mathbf 0  &  \text{on} \  {\Gamma:=\partial \Omega},\\
%\bs\sigma \mathbf{n}& = & \mathbf{g} &\text{on}\ {\Gamma_{N}},
\end{array}
\right.
\end{equation}
where   $\Omega\subset \mathbb{R}^2$ is a bounded polygonal domain, $\mathbf\sigma\in \mathbb{R}_{sym}^{2\times 2}$ denotes the symmetric stress tensor field, $\mathbf u\in \mathbb{R}^2$ the displacement flied,
$\epsilon(\mathbf u)=\frac{1}{2}\left( \nabla \mathbf u + (\nabla\mathbf u)^T \right)$ the strain tensor, $\mathbf f\in \mathbb{R}^2$ the body loading density, %$\mathbf g\in \mathbb{R}^2$ the surface traction, $\mathbf n$ the unite outward vector normal to $\Gamma_{N}$, 
and $\mathbb{C}$ the elasticity module tensor with
\begin{equation*}
\mathbb{C}\mathbf\epsilon(\mathbf u)=2\mu\mathbf\epsilon(\mathbf u)+\lambda \text{div}\mathbf u \mathcal{I}.
\end{equation*}
Here $\mathcal{I}$ is the $2\times 2$ identity tensor, $\text{tr}(\bs\sigma)$ the trace of the stress tensor $\mathbf\sigma$, and $\mu,\ \lambda$ the Lam\'{e} parameters.

We introduce some notations as follows. For an arbitrary open set $T$,  we  denote by $H^k(T)$ the  usual Sobolev space consisting
of functions defined on $T$ with derivatives of order up to $k$ being square-integrable, with norm $\|\cdot\|_{k,T}$ and semi-norm $|\cdot|_{k,T}$.  In particular,  $H^0(T)=L^2(T)$. When $T=\Omega$, we   abbreviate $\|\cdot\|_{k,\Omega}$ and $|\cdot|_{k,\Omega}$ to $\|\cdot\|_{k}$ and $|\cdot|_{k}$, respectively, and denote $\|\cdot\|:= \|\cdot\|_{0}$. We use the same notations of  norms and semi-norms as above  for corresponding vector or tensor spaces.   For any vector $\mathbf\alpha = (\alpha_i)_{i=1}^n\in\mathbb R^n$, we denote $\|\mathbf\alpha\|_{l^2}:=\left(\sum\limits_{i = 1}^n\alpha_i^2\right)^{\frac{1}{2}}$ and   $\|\mathbf\alpha\|_{l^{\infty}}:=\max\limits_{1\leq i\leq n}|\alpha_i|$.

Throughout the paper, we use notation $a\lesssim b$ (or $a\gtrsim b$) to
represent that there exists a constant $C$, independent of mesh size $h$
and the Lam\'e constant $\lambda$, such that $a\leq Cb$ (or $a\geq
Cb$), and use $a\approx b$ to denote $a\lesssim b\lesssim a$.

Define the spaces
$$
\Sigma:=\left\{\mathbf\tau\in L^2(\Omega;\mathbb{R}_{sym}^{2\times 2}),\ \int_{\Omega}\text{tr}(\mathbf\tau)\ =0 \right\},
$$
 $$\mathbf V:=(H^1_0(\Omega))^2=\{\mathbf v\in
(H^1(\Omega))^2:\ \mathbf v|_{\Gamma}=0\},$$
where $L^2(\Omega;\mathbb{R}_{sym}^{2\times 2})$  denotes the space
of square-integrable symmetric tensors, and $\text{tr}(\mathbf\tau):=\tau_{11}+\tau_{22}$   the trace of tensor $\tau$.
Then we have the following weak problem for  the system \eqref{eq:model-prob}: Find $(\mathbf \sigma,\mathbf u)\in
\Sigma\times \mathbf V$ such
that
\begin{align}
\label{eq:model-weak}\left\{
\begin{array}{llllll} \displaystyle a(\mathbf\sigma,\mathbf\tau)& + & b(\mathbf\tau,\mathbf u)& =& 0
&\text{for all }\ \mathbf\tau\in \mathbf\Sigma,\\
\displaystyle & &  b(\mathbf\sigma,\mathbf v)& = & F(\mathbf v) &
\text{for all }\ \mathbf v\in \mathbf V,\end{array}\right.
\end{align}
where
\begin{align*}
a(\mathbf\sigma,\mathbf\tau)&=\int_{\Omega}\mathbb{C}^{-1}\mathbf\sigma:\mathbf\tau= \frac{1}{2\mu}\int_{\Omega}\left(\mathbf\sigma:\mathbf\tau-\frac{\lambda}{2(\mu+\lambda)}\text{tr}(\mathbf\sigma)\text{tr}(\mathbf\tau)\right),\\
b(\mathbf\tau,\mathbf v)&=-\int_{\Omega}\mathbf\tau:\mathbf\epsilon(\mathbf
v),\quad
F(\mathbf v)=-\int_{\Omega}\mathbf f\cdot\mathbf
v.
\end{align*}

It is well-known that   the weak problem \eqref{eq:model-weak} admits a unique solution.
%
%The weak formulation of equation \eqref{eq:model-prob} is: Find $(\bs\sigma,\mathbf u)\in L^2(\Omega;\mathbb{R}_{sym}^{2\times 2})\times (H_{0,D}^{1}(\Omega))^2$ such that
%\begin{align}\label{weak formulation}
%\left\{\begin{array}{lll}
%a(\bs\sigma,\mathbf{\tau})-&b(\mathbf u,\bs\tau)=0  &\quad\text{for all}\ \bs\tau\in L^2(\Omega;\mathbb{R}_{sym}^{2\times 2});\\
%&b(\mathbf{v},\bs\sigma)=F(\mathbf v) &\quad\text{for all}\ \mathbf v\in (H_{0,D}^{1}(\Omega))^2,
%\end{array}
%\right.
%\end{align}
%where
%\begin{align}
%&a(\bs\sigma,\bs\tau):=\int_{\Omega}\bs\sigma:\mathbb{C}^{-1}\bs\tau dxdy,\\
%&b(\mathbf v,\bs\tau):=\int_{\Omega}\bs\tau:\epsilon(\mathbf v)dxdy,\\
%&F(\mathbf v):=\int_{\Omega}\mathbf{f}\cdot\mathbf vdxdy+\int_{\Gamma_{D}}\mathbf{g}\cdot\mathbf vds.
%\end{align}
%and $L^2(\Omega;\mathbb{R}_{sym}^{2\times 2})$ denotes the space of square-integrable symmetric tensors, $(H_{0,D}^{1}(\Omega))^2=\{ \mathbf v\in (H^1(\Omega))^2:\mathbf v|_{\Gamma_{D}}=0\}$.
%
%Since the Neumann boundary condition can be built into the weak formulation. For simplicity, in the following sections, we
% only consider the case $\Gamma_{N}=\emptyset$.

 \section{Hybrid stress finite element discretization}\label{sec:prelimilary}
\subsection{Geometric properties of quadrilateral  meshes}

Let $\{\mathcal
T_h\}_{h>0}$ be a partition of $\bar\Omega$ by convex quadrilaterals with the mesh size $h:=\max\limits_{K\in\mathcal T_h}h_K$, where
$h_K$ is   the diameter of quadrilateral  $K\in \mathcal T_h$.  Let $Z_i(x_i^K,y_i^K)$ and $\hat Z_i(\xi_i,\eta_i)$  for $ 1\leq i\leq 4$ be the  vertices of $K$ and  the reference
element  $\hat{K}=[-1,1]^2$  (cf. Figure \ref{Fig:bilinear}), respectively.
 There exits a unique invertible bilinear mapping
$F_K:\ \hat{K}\rightarrow K$ that maps $\hat K$ onto $K$ with $F_K(\hat Z_i)=Z_i$. The mapping $F_K$ is of the form
\begin{align}\label{eq:trans}
\left(\begin{array}{c} x \\  y\end{array}\right)= F_K(\xi,\eta)
=\left(\begin{array}{c}
a_0^K+a_1^K\xi+a_2^K\eta+a_{12}^K\xi\eta \\
b_0^K+b_1^K\xi+b_2^K\eta+b_{12}^K\xi\eta
\end{array}\right),
\end{align}
where  $\xi,\eta\in [-1,1]$ are the local coordinates and
$$
\left(\begin{array}{cc} a_0^K & b_0^K\\ a_1^K & b_1^K\\ a_2^K & b_2^K \\ a_{12}^K & b_{12}^K\end{array}\right)=
\frac{1}{4}\left(\begin{array}{rrrr}
 1 & 1 & 1 & 1\\
 -1& 1 & 1 &-1\\
 -1&-1 & 1 & 1\\
 1&-1 &1 &-1\\
\end{array}\right)
\left(\begin{array}{cc} x_1^K & y_1^K\\ x_2^K & y_2^K \\ x_3^K & y_3^K \\ x_4^K & y_4^K  \end{array}\right).
$$
In the following we may omit the superscript $K$ of the above notations  if there is no confusing.

%Let $\mathcal{T}_{h}$ be a partition of the domain $\Omega$ by  conventional quadrilaterals with the size $h:=\max\limits_{K\in\mathrm{T}_{h}}h_{K}$, $h_{K}$ the longest edge of $K$. Let $K$ be a convex quadrilateral with vertices $Z_{i}(x_{i},y_{i}), 1\leq i\leq 4$, $\hat{K}=[-1,1]\times[-1,1]$ be the reference square with vertices $\hat{Z}_{i}$. There exists a unique invertible bilinear mapping $F_{K}$ (see Figure \ref{Fig:bilinear}) that maps $\hat{K}$ onto $K$ such that $F_{K}(\hat{Z}_{i})=Z_{i},1\leq i\leq 4$ given by
%\[x=\sum_{i=1}^{4}x_{i}^{K}N_{i},\quad y=\sum_{i=1}^{4}y_{i}^{K}N_{i},\]
%where
%\[
%\begin{array}{ll}
%N_{1}=\frac{1}{4}(1-\xi)(1-\eta),&N_{2}=\frac{1}{4}(1+\xi)(1-\eta),\\
%N_{3}=\frac{1}{4}(1+\xi)(1+\eta),&N_{4}=\frac{1}{4}(1-\xi)(1+\eta).\\
%\end{array}\]
\begin{figure}[h]
\setlength{\unitlength}{0.8cm}
\begin{picture}(14,6)
%reference element
\put(1,1){\framebox(3.5,3.5)}\put(0,2.75){\vector(1,0){5.5}}\put(5.1,3){\footnotesize
$\xi$} \put(2.75,0){\vector(0,1){5.5}}\put(2.5,5.1){\footnotesize
$\eta$}
\put(2.5,4.55){1}\put(2.3,0.5){-1}\put(4.6,2.78){1}\put(0.5,2.78){-1}
\put(0.5,0.5){\footnotesize $\hat{Z}_1$}\put(4.5,0.5){\footnotesize
$\hat{Z}_2$} \put(4.5,4.5){\footnotesize
$\hat{Z}_3$}\put(0.5,4.5){\footnotesize $\hat{Z}_4$}
%element
\put(6,2.75){\vector(1,0){1}}\put(6.3,3){\footnotesize
$F_K$}
\put(8,2){\line(1,2){1.5}}
\put(9.5,5){\line(4,1){2}}\put(11.5,5.5){\line(1,-4){1.06}}\put(8,2){\line(6,-1){4.5}}
%diagonal

\put(8,2){\line(1,1){3.5}} \put(9.5,5){\line(4,-5){3}}

%middle point
\put(10.9,3.025){$\bullet$} \put(9.65,3.65){$\bullet$}

\put(9.75,3.75){\line(2,-1){1.3}}

\put(11.1,3.125){$O_2$}  \put(9.15,3.75){$O_1$} \put(10.1,3.1){$d_K$}

\put(7.2,1){\vector(1,0){6.2}}\put(7.5,0.7){\vector(0,1){5}}

\put(13.3,1.2){x}\put(7,5){y}
\put(7.7,1.6){\footnotesize $Z_1$}\put(12.7,1.3){\footnotesize
$Z_2$}\put(11.7,5){\footnotesize $Z_3$}\put(9.0,5){\footnotesize
$Z_4$}\put(10,2.5){\footnotesize $K$} \put(3,3){\footnotesize
$\hat{K}$}
\end{picture}
\caption{Bilinear transformation $F_K$ maps   reference element
$\hat K$ (in the left) to   element $K$ (in the
right).}\label{Fig:bilinear}
\end{figure}
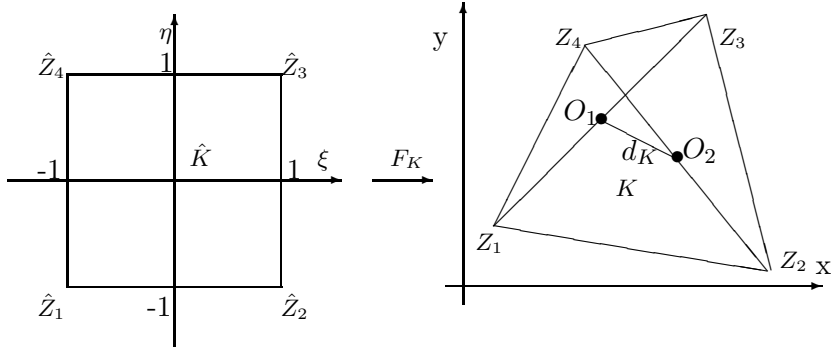

%\put(5,1){\line(1,2){1.5}}
%\put(6.5,4){\line(4,1){2}}\put(8.5,4.5){\line(1,-4){1.06}}\put(5,1){\line(6,-1){4.5}}
%%diagonal
%
%\put(5,1){\line(1,1){3.5}} \put(6.5,4){\line(4,-5){3}}
%
%%middle point
%\put(7.9,2.025){$\bullet$} \put(6.65,2.65){$\bullet$}
%
%\put(6.75,2.75){\line(2,-1){1.3}}
%
%\put(8.1,2.125){$O_2$}  \put(6.15,2.75){$O_1$} \put(7.1,2.1){$d_K$}

%We can also express the mapping $F_{K}$ as:
%\[ x=a_{0}+a_{1}\xi+a_{2}\eta+a_{12}\xi\eta,\quad y=b_{0}+b_{1}\xi+b_{2}\eta+b_{12}\xi\eta;\]
%where by suppressing the index $"K"$,
%\[\begin{array}{ll}
%a_{0}=(x_{1}+x_{2}+x_{3}+x_{4})/4,\quad&b_{0}=(y_{1}+y_{2}+y_{3}+y_{4})/4,\\
%a_{1}=(-x_{1}+x_{2}+x_{3}-x_{4})/4,\quad&b_{1}=(-y_{1}+y_{2}+y_{3}-y_{4})/4,\\
%a_{2}=(-x_{1}-x_{2}+x_{3}+x_{4})/4,\quad&b_{2}=(-y_{1}-y_{2}+y_{3}+y_{4})/4,\\
%a_{12}=(x_{1}-x_{2}+x_{3}-x_{4})/4,\quad&b_{12}=(y_{1}-y_{2}+y_{3}-y_{4})/4.
%\end{array}
%\]

%\smalls-kip
The Jacobi matrix and   Jacobian of $F_K$ are respectively given by
\begin{equation}\label{eq:Jacobi}
DF_K(\xi,\eta) = \left(\begin{array}{cc}
\frac{\partial x}{\partial\xi} & \frac{\partial x}{\partial\eta} \\
\frac{\partial y}{\partial\xi} & \frac{\partial y}{\partial\eta}
\end{array}\right) = \left(\begin{array}{cc}
a_1+a_{12}\eta & a_2+a_{12}\xi \\
b_1+b_{12}\eta & b_2+b_{12}\xi
\end{array}\right),
\end{equation}
\begin{equation}\label{J_K}
J_K(\xi,\eta) = \det (DF_K) = J_0 + J_1\xi + J_2\eta,
\end{equation}
where
$$
J_0 = a_1b_2 - a_2b_1,\quad J_1 = a_1b_{12} - a_{12}b_1,\quad J_2 = a_{12}b_2 - a_2b_{12}.
$$
It is easy to obtain the inverse of the Jacobi $DF_K$ with
\begin{equation}\label{eq:inv_Jacobi}
DF_K^{-1} \circ F_K(\xi,\eta) = \left(\begin{array}{cc} \frac{\partial\xi}{\partial x} & \frac{\partial\xi}{\partial y}\\
\frac{\partial\eta}{\partial x} & \frac{\partial\eta}{\partial y}  \end{array}\right)
=\frac{1}{J_K}\left(\begin{array}{cc} b_2 + b_{12}\xi & -a_2 -a_{12}\xi \\ -b_1 - b_{12}\eta & a_1 + a_{12}\eta\end{array}\right).
\end{equation}

%\begin{figure}[htdp]
%\setlength{\unitlength}{1cm}
%\begin{picture}(14,4)
%
%%element K
%\put(5,1){\line(1,2){1.5}}
%\put(6.5,4){\line(4,1){2}}\put(8.5,4.5){\line(1,-4){1.06}}\put(5,1){\line(6,-1){4.5}}
%%diagonal
%
%\put(5,1){\line(1,1){3.5}} \put(6.5,4){\line(4,-5){3}}
%
%%middle point
%\put(7.9,2.025){$\bullet$} \put(6.65,2.65){$\bullet$}
%
%\put(6.75,2.75){\line(2,-1){1.3}}
%
%\put(8.1,2.125){$O_2$}  \put(6.15,2.75){$O_1$} \put(7.1,2.1){$d_K$}
%
%\end{picture}
%\caption{Quadrilateral $K$ and two middle points $O_1$, $O_2$ of its diagonals.}\label{Fig:K}
%\end{figure}
%
Throughout this paper we assume the partition $\mathcal T_{h}$ is  shape regular  in the following sense \cite{ZM-Z1}:    There exist a constant $\varrho>2$, independent
of $h$, such that for all $K\in \mathcal{T}_{h}$ it holds
\begin{equation}\label{partition condition}
 h_{K}\leq \varrho \rho_{K}.
\end{equation}
Here
%\begin{equation*}
 $ \rho_{K}:=\min\limits_{1\leq i\leq4}\  \rho_i$, with $\rho_i$ being   the diameter of the largest circle  inscribed
 in $T_{i}$, the sub-triangle
 of $K$ with vertices $Z_{i-1}$, $Z_{i}$ and $Z_{i+1}$
 (the index on $Z_{i}$ is modulo 4) for $i=1,\cdots,4$.

We   introduce several additional mesh conditions which will be used in the forthcoming analysis of superconvergence (Section \ref{super}).

\begin{itemize}

\item \textbf{(MC1) Diagonal condition}: There exists a constant $\alpha>0$ such that for any quadrilateral  $K\in \mathcal{T}_{h}$, the distance, $d_K$ ($d_K=|O_{1}O_{2}|=\frac{1}{2}\sqrt{a_{12}^2+b_{12}^2}$), between the
midpoints of the diagonals of $K$   (See Figure \ref{Fig:bilinear}) satisfies
\begin{equation}\label{d_K}
 d_K= O(  h_{K}^{1+\alpha}).
 \end{equation}
 %$$
%d_{K}=|O_{1}O_{2}|=O(h_{K}^{1+\alpha})\Longleftrightarrow\max\{|a_{12}|,|b_{12}|\}=O(h_{K}^{1+\alpha}),\ \alpha>0.
%$$

\item \textbf{(MC2) Neighboring condition}:  For any two quadrilaterals $K_{1},\ K_{2}\in\mathcal{T}_{h}$
sharing a common edge, it holds, for $j=1,2$,
\begin{equation}
%\footnotesize
a_{j}^{K_{1}}=a_{j}^{K_{2}}(1+O(h_{K_{1}}^{\alpha}+h_{K_{2}}^{\alpha})),\quad b_{j}^{K_{1}}=b_{j}^{K_{2}}(1+O(h_{K_{1}}^{\alpha}+h_{K_{2}}^{\alpha})).
\end{equation}

%\item \textbf{(MC3) Condition $\alpha$} (\cite{ZM-Z3}):  $\mathcal{T}_{h}$ satisfies   \textbf{Diagonal condition} (MC1)  as well as  \textbf{Neighboring condition} (MC2).

\end{itemize}

\begin{remark} \textbf{Diagonal condition} (MC1) is also called $(1+\alpha)$-section condition (cf. \cite{Ming-Shi}).
Note that $K$ is a parallelogram if and only if $d_{K}=0$, which means $\alpha=+\infty$.  When $\alpha=1$,  (MC1) is  the Bi-Section Condition or  condition B  \cite{Shi}.
\end{remark}

\begin{remark} $\mathcal{T}_{h}$ is said to satisfy Jamet condition \cite{Jamet1977} if there exists a constant $r > 0$ such that
$h_K\leq r \tilde\rho_K $
holds for any quadrilateral  $K\in \mathcal{T}_{h}$, where $\tilde\rho_K $ is   the diameter of the largest circle
inscribed in $K$. As shown in \cite{Ming-Shi},   if both Jamet condition and \textbf{Diagonal condition} (MC1)  hold, then
  $ \mathcal{T}_{h}$ is shape regular for sufficiently smal $h$.
\end{remark}

In view of the shape regularity condition \eqref{partition condition}, it is easy to obtain the following estimates for  the Jacobian  $J_K$ given in \eqref{J_K}.
\begin{lemma}\label{JK-J1-J2}
For any $K\in\mathcal{ T}_{h}$ it holds
\begin{equation}\label{lem:J0_J1_J2-1}
J_K\approx J_0\approx h_{K}^2.
\end{equation}
Further more, if \textbf{Diagonal condition} (MC1) is satisfied, then it holds
\begin{equation}\label{lem:J0_J1_J2-2}
 \max\{|J_1|, |J_2|\}\approx h_{K}^{2+\alpha}.%, \quad J_2\lesssim h_{K}^{2+\alpha}.
\end{equation}

\end{lemma}
%\begin{equation}
%\frac{J_{0}}{J_{K}}=1+O(h_{K}^{\alpha}),\quad\frac{J_{1}}{J_{K}}=O(h_{K}^{\alpha}),\quad \frac{J_{2}}{J_{K}}=O(h_{K}^{\alpha}).
%\end{equation}

%we say that they satisfy a neighboring condition.

%\begin{definition}\label{def:neighber_condition}
%A partition $\mathcal{T}_{h}$ is said to satisfy \textbf{condition $\alpha$} if there exists $\alpha>0$ such that:
%(i) any $K\in\mathcal{T}_{h}$ satisfies the diagonal condition, and (ii) any two $K_{1},\ K_{2}$ in $\mathcal{T}_{h}$ that
%share a common edge satisfy a \textbf{neighboring condition}: for $j=1,2$
%\begin{align}
%a_{j}^{K_{1}}=a_{j}^{K_{2}}(1+O(h_{K_{1}}^{\alpha}+h_{K_{2}}^{\alpha})),\quad b_{j}^{K_{1}}=b_{j}^{K_{2}}(1+O(h_{K_{1}}^{\alpha}+h_{K_{2}}^{\alpha})).
%\end{align}
%\end{definition}
%\begin{itemize}
%\item \textbf{RDP(N,$\Psi$)} (\cite{Acosta}): $\mathcal{T}_{h}$ is said to satisfy the regular decomposition property with constants $N\in R$ and $0<\Psi<\pi$, or shortly \textbf{RDP(N,$\Psi$)},  if  any $K\in \mathcal{T}_{h}$ can be divided into two triangles along one of its diagonal, which will always be called $d_{1}$ with the other one being called   $d_{2}$, in such a way that $|d_{1}|/|d_{2}|\leq N$ and both triangles satisfy the maximum angle condition, i.e. each interior angle of these two triangles is less than or equal to  $\Psi$.
%\end{itemize}

\subsection{Pian-Sumihara's hybrid stress finite element method}

In view of the mapping $F_K$, for  any function $\hat{w}(\xi,\eta)$   on $\hat{K}$ we define  function $w(x,y)$  on $K\in \mathcal{T}_{h}$ with
$$
w(x,y):=\hat{w}(\xi,\eta) \quad  \text{or equivalently }\quad w:=\hat{w}\circ F_{K}^{-1}.
$$

In Pian-Sumihara's hybrid stress finite element (abbr. PS element) method  \cite{P-S}  for the problem \eqref{eq:model-weak}, continuous piecewise isoparametric bilinear interpolation is used for the approximation of displacement, namely the displacement approximation space $\mathbf V_{h}\subset \mathbf  V$ is taken as
$$
\mathbf V_{h}:=S_{h}\times S_{h}
$$
with
 \begin{align*}
S_{h}=\{ v\in H_{0}^1(\Omega): \hat{v}=v|_{K}\circ F_{K} \in \text{ span}\{1,\xi,\eta,\xi\eta\} ,\ \text{for all}\ K\in \mathcal{T}_{h}\}.
\end{align*}

%Some  calculations show
%{\small
%\begin{eqnarray}
%\label{eq:scaling} \frac{\partial^{r}\hat{v}}{\partial\xi^r} & = &\left((a_1+a_{12}\eta)\frac{\partial}{\partial x}+(b_1+b_{12}\eta)\frac{\partial}{\partial y}\right)^{r}v,\quad r=1,2,\\
% \frac{\partial^{2}\hat{v}}{\partial\xi\partial\eta} & = & a_{12}\frac{\partial v}{\partial x}+b_{12}\frac{\partial v}{\partial y}
%+\left((a_1+a_{12}\eta)\frac{\partial}{\partial x}+(b_1+b_{12}\eta)\frac{\partial}{\partial y}\right) \nonumber\\
%& &\qquad \times \left((a_2+a_{12}\xi)\frac{\partial}{\partial x}+(b_2+b_{12}\xi)\frac{\partial}{\partial y}\right)v.\label{eq:scaling2}
%\end{eqnarray}
%}

To describe  the stress approximation of PS element, we abbreviate the symmetric tensor $\mathbf
\tau=\left(\begin{array}{cc}\tau_{11} & \tau_{12}\\ \tau_{12} &
\tau_{22} \end{array}\right)$ to $\mathbf
\tau=(\tau_{11},\tau_{22},\tau_{12})^T$.  The stress mode of PS element is of  the following form on $\hat{K}$:
\begin{equation}\label{stress-ps}
\hat{\mathbf \tau}=\left(\begin{array}{c}\hat{\mathbf \tau}_{11}\\ \hat{\mathbf \tau}_{22}\\
\hat{\mathbf \tau}_{12} \end{array}\right)=\left(\begin{array}{ccccc} 1
&0 & 0 & \eta & \frac{a_2^2}{b_2^2}\xi \\ 0 & 1 & 0 &
\frac{b_1^2}{a_1^2}\eta & \xi \\ 0 & 0 & 1 & \frac{b_1}{a_1}\eta &
\frac{a_2}{b_2}\xi \end{array}\right)\mathbf\beta^{\mathbf \tau} =: A\mathbf\beta^{\mathbf\tau},\
\beta^{\mathbf \tau}\in\mathbb R^5.
\end{equation}
Then   the
corresponding stress approximation space, $\Sigma_h\subset\Sigma$, for PS element is given by
\begin{equation}
\Sigma_h:=\{\mathbf \tau\in\Sigma:\,\hat{\mathbf \tau}=\mathbf \tau|_K\circ
F_K\,\text{ is of the form  } (\ref{stress-ps}) \text{ for all }
K\in\mathcal T_h\}.
\end{equation}
As a result,  the PS
  element method for the problem \eqref{eq:model-weak} is given as follows. Find $(\mathbf \sigma_h,\mathbf u_h)\in \Sigma_h\times \mathbf V_h$ such that
\begin{equation}\label{eq:HybridFEM}
\left\{ \begin{array}{llllll} a(\mathbf \sigma_h,\mathbf\tau) & + & b(\mathbf\tau,\mathbf u_h) & = & \mathbf 0 & \text{for all } \mathbf\tau\in\Sigma_h,\\
& & b(\mathbf\sigma_h,\mathbf v) & = & F(\mathbf v) & \text{for all }\mathbf v\in\mathbf V_h.  \end{array}  \right.
\end{equation}

Let $(\bs\sigma,\mathbf u)\in \left(H^{1}(\Omega, \mathbb{R}_{sym}^{2\times 2})\bigcap \Sigma\right) \times \left(\mathbf V\bigcap (H^{2}(\Omega))^2\right)$ be the solution of the problem \eqref{eq:model-weak}. It has been shown  in \cite{Yu-Xie-Carsten} that the following uniform error estimate holds for the PS
  element method:
\begin{equation}\label{a priori est}||\mathbf\sigma-\mathbf\sigma_h||+ |\mathbf u-\mathbf u_{h}|_{1}    \lesssim   h \left(||\mathbf u||_{2}+||\mathbf\sigma||_{1}\right).
\end{equation}

\section{Superconvergence analysis}\label{super}

\subsection{Preliminary results}

%Let $\nabla v=(\frac{\partial v}{\partial x},\frac{\partial v}{\partial y})^{t}$,

%\begin{definition}\cite{Acosta}
%$\mathcal{T}_{h}$ is regular with constant $N\in R$ and $0<\Psi<\pi$, or shortly $RDP(N,\Psi)$, if we can
%divide $K$ into two triangles along one of its diagonal, which will always be called $d_{1}$, the other is $d_{2}$ in such a way that $|d_{1}|/|d_{2}|\leq N$ and both triangles satisfy the maximum angle condition, i.e. each interior angle of these two triangles is bounded from above by $\Psi$.
%\end{definition}
%\begin{lemma}
%Under the neighboring condition, we have
%\[ J_{0}|_{K_{1}}=J_{0}|_{K_{2}}(1+O(h_{K_{1}}^{\alpha}+h_{K_{2}}^{\alpha})).\]
%\end{lemma}
%\begin{lemma} \label{lem:J0_J1_J2}
%Let  $K\in \mathcal{T}_{h}$ satisfy   \textbf{Jamet condition} and \textbf{Diagonal condition}. Then it holds
%\[\frac{J_{0}}{J_{K}}=1+O(h_{K}^{\alpha}),\quad\frac{J_{1}}{J_{K}}=O(h_{K}^{\alpha}),\quad \frac{J_{2}}{J_{K}}=O(h_{K}^{\alpha}).\]
%\end{lemma}

We recall
$
v(x,y):=\hat{v}(\xi,\eta)=\hat{v}\circ F_{K}^{-1}(x,y).$
Some  calculations show
{\small
\begin{eqnarray}
\label{eq:scaling} \frac{\partial^{r}\hat{v}}{\partial\xi^r} & = &\left((a_1+a_{12}\eta)\frac{\partial}{\partial x}+(b_1+b_{12}\eta)\frac{\partial}{\partial y}\right)^{r}v,\quad r=1,2,\\
 \frac{\partial^{2}\hat{v}}{\partial\xi\partial\eta} & = & a_{12}\frac{\partial v}{\partial x}+b_{12}\frac{\partial v}{\partial y}
+\left((a_1+a_{12}\eta)\frac{\partial}{\partial x}+(b_1+b_{12}\eta)\frac{\partial}{\partial y}\right) \nonumber\\
& &\qquad \times \left((a_2+a_{12}\xi)\frac{\partial}{\partial x}+(b_2+b_{12}\xi)\frac{\partial}{\partial y}\right)v.\label{eq:scaling2}
\end{eqnarray}
}
In light of these two relations and Lemma \ref{JK-J1-J2}, we easily derive the following lemma.
\begin{lemma}\label{lem:dv_d_xi}
For all $K\in \mathcal T_h$ and  $v\in H^2(K)$,    it holds
\begin{eqnarray}
\left\|\frac{\partial\hat{v}}{\partial\xi}\right\|_{0,\hat{K}}+
\left\|\frac{\partial\hat{v}}{\partial\eta}\right\|_{0,\hat{K}}&\lesssim& |v|_{1,K},\label{1-norm esti}\\
\left\|\frac{\partial^2 \hat{v}}{\partial\xi^2}\right\|_{0,\hat{K}} + \left\|\frac{\partial^2 \hat{v}}{\partial\eta^2}\right\|_{0,\hat{K}}  &\lesssim & h_{K}|v|_{2,K}.\label{2-norm esti}
%\left\|\frac{\partial^2 \hat{v}}{\partial\xi\partial\eta}\right\|_{0,\hat{K}} & \lesssim & (h_{K}^{\alpha}|v|_{1,K}+h_{K}|v|_{2,K}).
\end{eqnarray}
 In particular,  if $\mathcal T_h$ satisfies \textbf{Diagonal condition} (MC1), then it holds
\begin{eqnarray}
\left\|\frac{\partial^2 \hat{v}}{\partial\xi\partial\eta}\right\|_{0,\hat{K}} & \lesssim & h_{K}^{\alpha}|v|_{1,K}+h_{K}|v|_{2,K}.
\end{eqnarray}
\end{lemma}

Let $\mathbf u^{I}\in\mathbf V_{h}$ be  the piecewise isoparametric bilinear interpolation of   $\mathbf u\in \mathbf V\bigcap (H^2(\Omega))^2$, then  it holds    the following estimate:
\begin{align} \label{eq:optimal estimate}
||\mathbf u-\mathbf u^{I}||_{0,K}+h_{K}|\mathbf u-\mathbf u^{I}|_{1,K}\lesssim h_{K}^2||u||_{2,K}, \quad  \text{for all}\ K\in \mathcal{T}_{h}.
\end{align}
%holds for any $K\in \mathcal T_h$ under \textbf{Shape regularity condition} (MC1).

%\begin{remark}
%$RDP(N,\Psi)$ is a sufficient and necessary condition for \eqref{eq:optimal estimate} to be hold. In \cite{Acosta}, the author proved the sufficient condition, then Ming and Shi confirmed necessary condition by a counter-example\cite{Ming-Shi}.
%\end{remark}

Let $\mathbf\sigma^{I}\in \Sigma_{h}$ be the projection of $\mathbf \sigma\in\Sigma$  in the $a(\cdot,\cdot)$-inner product, namely $\mathbf\sigma^{I}$ satisfies
\begin{align}\label{eq:interpolation}
a(\mathbf \sigma^I,\mathbf\tau) = a(\mathbf\sigma,\mathbf\tau)\qquad\text{for all }\mathbf\tau\in\Sigma_h.
\end{align}
Thanks to \eqref{stress-ps} and \eqref{eq:interpolation}, we obtain, for all $K\in \mathcal{T}_{h},$
\begin{eqnarray}\label{eq:sigma_I_K}
\mathbf\sigma^I|_K  = AH_K^{-1}\int_KA^T\mathbb C^{-1}\mathbf\sigma  \text{ with } H_K := \int_KA^T\mathbb C^{-1}A ,
\end{eqnarray}
\begin{equation}\label{lem:L2orth}
\int_K (\mathbf\sigma - \mathbf\sigma^I)  = {\bf 0}.%\qquad \text{ for all }\mathbf\vartheta\in C_K.
\end{equation}
%By the definition of $\bs\sigma^I$, we have the following lemmas.
In addition, we have  the following lemma.

%\begin{lemma}\label{lem:L2orth}
%For any $\bs\sigma\in\Sigma$, it holds
%$$
%\int_K (\bs\sigma - \bs\sigma^I)\cdot \mathbf\vartheta  = 0\qquad \text{ for all }\mathbf\vartheta\in C_K.
%$$
%\end{lemma}
%\begin{proof}
%For any $\mathbf\vartheta\in C_K$, it is easy to check that $\mathbb C\mathbf\vartheta\in\Sigma_h$. By the definition of $\bs\sigma^I$, we have
%\begin{eqnarray*}
%\int_K (\bs\sigma - \bs\sigma^I) \cdot \mathbf\vartheta & = & \int_K(\bs\sigma - \bs\sigma^I)\cdot\mathbb C^{-1} \left( \mathbb C\mathbf\vartheta \right) = 0.
%\end{eqnarray*}
%\end{proof}

\begin{lemma}\label{lem:L2app_sigmaI}
Under  \textbf{Diagonal condition} (MC1), for all $ K\in \mathcal{T}_{h}$ it holds
\begin{equation}\label{sigma-inter}
\|\mathbf\sigma - \mathbf\sigma^I\|_{0,K} \lesssim h_K\|\mathbf\sigma\|_{1,K},
\end{equation}
 \begin{equation}\label{sigma-inter2}
\left|\int_{\hat K}(\hat{\mathbf\sigma} - \hat{\mathbf\sigma^I})\right| \lesssim h_K^{\alpha} \|\mathbf\sigma\|_{1,K}   .
\end{equation}
\end{lemma}
\begin{proof}
Let $\tilde{\mathbf\sigma}^I\in\Sigma_h$ be the $L^2-$projection of $\mathbf\sigma$ with
\begin{equation*}
\int_{\Omega}\tilde{\mathbf\sigma}^I\cdot\mathbf\tau  = \int_{\Omega}\mathbf\sigma\cdot\mathbf\tau ,\qquad\text{for all }\mathbf\tau\in\Sigma_h.
\end{equation*}
Then we have
\begin{eqnarray}
\|\bs\sigma - \tilde{\bs\sigma}^I\|  \lesssim  h\|\mathbf\sigma\|_1,
\end{eqnarray}
and
\begin{eqnarray}\label{tilde-sigma}
\tilde{\mathbf\sigma}^{I}|_K  =  A \tilde H_K^{-1}\int_KA^T\mathbf\sigma  \text{ with } \tilde H_K := \int_K A^TA
\end{eqnarray}
for all $ K\in\mathcal T_h$. By triangle inequality, it holds
\begin{equation}\label{sigma-I}
\|\mathbf\sigma - \mathbf\sigma^I\| \leq \|\mathbf\sigma - \tilde{\mathbf\sigma}^I\| + \|\tilde{\mathbf\sigma}^{I} - \mathbf\sigma^I\| \lesssim h\|\mathbf\sigma\|_1 + \|\tilde{\mathbf\sigma}^{I} - \mathbf\sigma^I\|.
\end{equation}

We turn to estimate $\|\tilde{\mathbf\sigma}^{I} - \mathbf\sigma^I\|=\left(\sum_{K\in\mathcal T_h}\|\tilde{\mathbf\sigma}^{I} - \mathbf\sigma^I\|_{0,K}^2\right)^{1/2} $. In view of \eqref{eq:interpolation} and \eqref{tilde-sigma}, some   calculations yield
\begin{eqnarray}\label{aa}
H_K^{-1}A^T\mathbb C^{-1}
%& = & \frac{1}{4J_0}\left( \begin{array}{ccc} 1  &  0 & 0 \\ 0 & 1 & 0 \\ 0 & 0 & 1 \\ \\ \frac{3\left( 1 - \frac{\lambda}{2\mu + \lambda}\frac{b_1^2}{a_1^2} \right)\eta}{\left( 1+ \frac{b_1^2}{a_1^2} \right)^2}     &   \frac{3\left( \frac{b_1^2}{a_1^2} - \frac{\lambda}{2\mu + \lambda} \right)\eta}{\left( 1+ \frac{b_1^2}{a_1^2} \right)^2}  & \frac{12(\mu + \lambda)}{2\mu + \lambda} \frac{\frac{b_1}{a_1}}{\left( 1+ \frac{b_1^2}{a_1^2} \right)^2}\eta \\ \\
%\frac{3\left( \frac{a_2^2}{b_2^2} - \frac{\lambda}{2\mu + \lambda} \right)\xi}{\left( 1+ \frac{a_2^2}{b_2^2} \right)^2}        &
%\frac{3\left( 1 - \frac{a_2^2}{b_2^2}\frac{\lambda}{2\mu + \lambda} \right)\xi}{\left( 1+ \frac{a_2^2}{b_2^2} \right)^2}        &
%\frac{12(\mu + \lambda)}{2\mu + \lambda} \frac{\frac{a_2}{b_2}}{\left( 1+ \frac{a_2^2}{b_2^2} \right)^2} \xi
%\end{array}\right) + h.o.t.\\
&=& \frac{1}{4J_0}\left( \begin{array}{ccc} 1  &  0 & 0 \\ 0 & 1 & 0 \\ 0 & 0 & 1 \\ \\ d_{41}\eta     &    d_{42}\eta  & d_{43}\eta \\ \\
d_{51}\xi        &
d_{52}\xi       &
d_{53} \xi
\end{array}\right) + h.o.t.,
\end{eqnarray}
\begin{eqnarray}
\tilde H_K^{-1} A^T & = & \frac{1}{4J_0} \left( \begin{array}{ccc}
1 & 0 & 0 \\ 0 & 1 & 0 \\ 0 & 0 & 1 \\ \\
\frac{3\eta}{1 + \frac{b_1^2}{a_1^2} + \frac{b_1^4}{a_1^4}}   &   \frac{3\frac{b_1^2}{a_1^2}\eta}{1 + \frac{b_1^2}{a_1^2} + \frac{b_1^4}{a_1^4}}  &
\frac{3\frac{b_1}{a_1}\eta}{1 + \frac{b_1^2}{a_1^2} + \frac{b_1^4}{a_1^4}} \\ \\
\frac{3\frac{a_2^2}{b_2^2}\xi}{1 + \frac{a_2^2}{b_2^2} + \frac{a_2^4}{b_2^4}}  &  \frac{3\xi}{1 + \frac{a_2^2}{b_2^2} + \frac{a_2^4}{b_2^4}}  &
\frac{3\frac{a_2}{b_2}\xi}{1 + \frac{a_2^2}{b_2^2} + \frac{a_2^4}{b_2^4}}
\end{array}  \right) + h.o.t.,
\end{eqnarray}
where
\begin{align*}
d_{41}&=\frac{3\left( 1 - \frac{\lambda}{2\mu + \lambda}\frac{b_1^2}{a_1^2} \right)}{\left( 1+ \frac{b_1^2}{a_1^2} \right)^2},\quad d_{42}=\frac{3\left( \frac{b_1^2}{a_1^2} - \frac{\lambda}{2\mu + \lambda} \right)}{\left( 1+ \frac{b_1^2}{a_1^2} \right)^2},\quad d_{43}=\frac{12(\mu + \lambda)}{2\mu + \lambda} \frac{\frac{b_1}{a_1}}{\left( 1+ \frac{b_1^2}{a_1^2} \right)^2} ,\\
d_{51}&=\frac{3\left( \frac{a_2^2}{b_2^2} - \frac{\lambda}{2\mu + \lambda} \right)}{\left( 1+ \frac{a_2^2}{b_2^2} \right)^2},\quad
d_{52}=\frac{3\left( 1 - \frac{a_2^2}{b_2^2}\frac{\lambda}{2\mu + \lambda} \right)}{\left( 1+ \frac{a_2^2}{b_2^2} \right)^2},\quad
d_{53}=\frac{12(\mu + \lambda)}{2\mu + \lambda} \frac{\frac{a_2}{b_2}}{\left( 1+ \frac{a_2^2}{b_2^2} \right)^2},
\end{align*}
and in each of the above two relations $h.o.t$ denotes a  different  higher-order-term matrix of the form
\begin{equation}\label{hot}
h.o.t=\frac{1}{J_0}(\tilde h_{ij}(\xi,\eta))_{5\times3}\text{ with } \max\limits_{i,j}\max\limits_{-1\leq\xi,\eta\geq1} |\tilde h_{ij}| \lesssim  h_K^\alpha.
\end{equation}
Obviously, it holds \begin{equation}\label{dij}
\max\limits_{4\leq i\leq 5,\ 1\leq j\leq 3}|d_{ij}|\approx 1.
\end{equation}

%where $d_{ij}(4\leq i\leq 5,\ 1\leq j\leq 3)$  is independent on $h_K$ and $\lambda$.
Denote $Q_K\mathbf\sigma: = \frac{1}{|K|}\int_K\mathbf\sigma$, then a combination of \eqref{lem:J0_J1_J2-1}, \eqref{lem:J0_J1_J2-2}, \eqref{eq:interpolation} and  \eqref{tilde-sigma}-\eqref{hot}  leads to
%\begin{eqnarray*}
%\| H_K^{-1}\int_KA^T\mathbb C^{-1} \mathbf\sigma & - & \tilde H_K^{-1}\int_KA^T\mathbf\sigma \|_{l^2}^2  \approx  \|\int_{\hat K}J_0\left( H_K^{-1}A^T\mathbb C^{-1} - \tilde H_K^{-1}A^T \right)\mathbf\sigma   \|_{l^2}^2 \\
%& = & \|\int_{\hat K}J_0\left( H_K^{-1}A^T\mathbb C^{-1} - \tilde H_K^{-1}A^T \right)(\mathbf\sigma - Q_K\mathbf\sigma)   \|_{l^2}^2 \\
%& \lesssim & \|\mathbf\sigma\|_{1,K}^2.
%\end{eqnarray*}
%Thus,
%\begin{eqnarray*}
%\|\mathbf\sigma^I - \tilde{\mathbf\sigma}^I\|_{0,K}^2 & \approx & h^2 \| H_K^{-1}\int_KA^T\mathbb C^{-1} \mathbf\sigma  -  \tilde H_K^{-1}\int_KA^T\mathbf\sigma \|_{l^2}^2 \\
%& \lesssim & h^2\|\mathbf\sigma\|_{1,K}^2.
%\end{eqnarray*}
\begin{eqnarray*}
\|\mathbf\sigma^I - \tilde{\mathbf\sigma}^I\|_{0,K} &=& \| A \int_K\left(H_K^{-1}A^T\mathbb C^{-1}- \tilde H_K^{-1}A^T\right) \mathbf\sigma  \|_{0,K} \\
&\approx& h_K|\int_K \left(H_K^{-1}A^T\mathbb C^{-1}- \tilde H_K^{-1}A^T\right) \mathbf\sigma| \\
&\approx&h_K | \int_K\left(H_K^{-1}A^T\mathbb C^{-1}- \tilde H_K^{-1}A^T\right) (\mathbf\sigma-Q_K\mathbf\sigma)  |\\
& \lesssim & h\|\mathbf\sigma\|_{1,K},
\end{eqnarray*}
which, together with \eqref{sigma-I}, indicates the desired result \eqref{sigma-inter}.

The thing left is to prove \eqref{sigma-inter2}.
%\end{proof}
%
%
%\begin{lemma}\label{lem:approx_on_hat_K_sigma}
%Under the diagonal condition, for any $\mathbf\sigma\in\Sigma$ and $K\in\mathcal T_h$, it holds
%\begin{equation*}
%\left|\int_{\hat K}(\mathbf\sigma - \mathbf\sigma^I)\right| \lesssim h_K^{\alpha} \|\mathbf\sigma\|_{1,K}.
%\end{equation*}
%\end{lemma}
%\begin{proof}
From \eqref{lem:L2orth} it follows
\begin{eqnarray*}
0 & = & \int_K(\mathbf\sigma - \mathbf\sigma^I) = \int_{\hat K}J_K(\mathbf{\hat\sigma} - \mathbf{\hat{\sigma^I}}) \\
&  = & J_0\int_{\hat K}(\mathbf{\hat\sigma} - \bs\hat{\mathbf{\sigma}^I}) +  J_1\int_{\hat K}\xi(\mathbf{\hat\sigma} - \mathbf{\hat{\sigma^I}}) + J_2\int_{\hat K}\eta(\mathbf{\hat\sigma} - \mathbf{\hat{\sigma^I}}),
\end{eqnarray*}
which, together with \eqref{lem:J0_J1_J2-1}-\eqref{lem:J0_J1_J2-2} and \eqref{sigma-inter}, implies
\begin{eqnarray*}
\left| \int_{\hat K}(\mathbf{\hat\sigma} - \mathbf{\hat{\sigma^I}}) \right| & \leq & \left|\frac{J_1}{J_0} \int_{\hat K}\xi (\mathbf{\hat\sigma} - \mathbf{\hat{\sigma^I}})  \right| + \left| \frac{J_2}{J_0}\int_{\hat K}\eta (\mathbf{\hat\sigma} - \mathbf{\hat{\sigma^I}})  \right| \\
& \lesssim & h_K^{\alpha} \|\mathbf{\hat\sigma} - \mathbf{\hat{\sigma^I}}\|_{0,\hat K}\\
& \lesssim& h_K^{\alpha} h_K^{-1}\|\mathbf\sigma - \mathbf{\sigma^I}\|_{0,K}\\
& \lesssim & h_K^{\alpha}\|\mathbf\sigma\|_{1,K}.
\end{eqnarray*}
\end{proof}

For any $K\in\mathcal T_h$, we follow \cite{ZM-Z1} to
define the modified partial derivatives $\frac{\tilde{\partial} v}{\partial x},\ \frac{\tilde{\partial} v}{\partial y}$ and the modified strain tensor $\tilde{\mathbf\epsilon}(\mathbf v)$ as
\begin{eqnarray*}
(J_{K}\frac{\tilde{\partial} v}{\partial x}|_{K}\circ F_{K})(\xi,\eta) & = & \frac{\partial y(0,0)}{\partial \eta}\frac{\partial \hat{v}}{\partial \xi}-\frac{\partial y(0,0)}{\partial \xi}\frac{\partial \hat{v}}{\partial \eta}=b_2\frac{\partial \hat{v}}{\partial \xi}-b_1\frac{\partial \hat{v}}{\partial \eta},\\
(J_{K}\frac{\tilde{\partial} v}{\partial y}|_{K}\circ F_{K})(\xi,\eta) &  = & -\frac{\partial x(0,0)}{\partial \eta}\frac{\partial \hat{v}}{\partial \xi}+\frac{\partial x(0,0)}{\partial \xi}\frac{\partial \hat{v}}{\partial \eta}=-a_2\frac{\partial \hat{v}}{\partial \xi}+a_1\frac{\partial \hat{v}}{\partial \eta},
\end{eqnarray*}
\begin{equation}\label{modified strain}
\tilde{\epsilon}(\mathbf v)= \left(\begin{array}{cc}
\frac{\tilde{\partial} v_1}{\partial x}&\quad\frac{1}{2}(\frac{\tilde{\partial} v_1}{\partial y}+\frac{\tilde{\partial} v_2}{\partial x})\\ \\
\frac{1}{2}(\frac{\tilde{\partial} v_1}{\partial y}+\frac{\tilde{\partial} v_2}{\partial x})&\quad\frac{\tilde{\partial} v_2}{\partial y}
\end{array}\right),
\end{equation}
respectively. By the definition of $\tilde{\mathbf\epsilon}(\mathbf v)$ it is easy to derive the following result.
%By the definition of $\tilde{\mathbf\epsilon}(\mathbf v)$, we have the following lemma.
\begin{lemma}\label{lem:epsilon_tilde}
Under  \textbf{Diagonal condition} (MC1),   for all $\mathbf v\in \mathbf V_h$ and $ K\in \mathcal{T}_{h}$ it holds
$$
\|\mathbf\epsilon(\mathbf v) - \tilde{\mathbf\epsilon}(\mathbf v)\|_{0,K} \lesssim h_K^{\alpha}|\mathbf v|_{1,K}.
$$
\end{lemma}
%\begin{proof}
%For any $\mathbf v\in \mathbf V_h$ and $K\in\mathcal T_h$, by scaling technique, we have
%\begin{eqnarray*}
%\|\mathbf\epsilon(\mathbf v) - \tilde{\mathbf\epsilon}(v)\|_{0,K}^2 & \lesssim & \frac{\max\{a_{12}^2,b_{12}^2 \}}{\min\limits_{(\xi,\eta)\in \hat K}J_K(\xi,\eta)}|\mathbf v|_{1,\hat K}^2 \lesssim h_K^{2\alpha}|\mathbf v|_{1,K}.
%\end{eqnarray*}
%
%
%\end{proof}
%

Define the bubble function space $\mathbf V_h^b$ as
\begin{align*}
\mathbf V_{h}^{b}:=\left\{\mathbf v^{b}\in (L^{2}(\Omega))^2:\hat{\mathbf v}^{b}(\xi,\eta)=\mathbf v^{b}|_{K}\circ F_{K}\in \text{span}\{\xi^2-1,\eta^2-1\}^2,\ \text{ for all} \ K\in\mathcal{T}_{h}\right\}.
\end{align*}
Then it is easy to verify that the PS stress mode \eqref{stress-ps} satisfies the relation (see \cite{Piltner})
\begin{equation}\label{PS-relation}
\int_{K}\mathbf{\tilde{\epsilon}}(\mathbf v^{b})\cdot\mathbf\tau =0,  \text{ for all }\mathbf v^{b}\in \mathbf V_{h}^{b}, \mathbf\tau\in\Sigma_h, K\in\mathcal{T}_{h}.
\end{equation}

\subsection{Superconvergence analysis}

Define two functions
$$E(\xi):=\frac{1}{2}(\xi^2-1),\quad F(\eta):=\frac{1}{2}(\eta^2-1).$$ Obviously it holds
%\begin{lemma}\label{lem:E_F}
%Functions $E(\xi)$ and $F(\eta)$ have the following properties
\begin{align}\label{E-F}
E^{'}(\xi)=\xi,\ \ E^{''}(\xi)=1, \ \ F^{'}(\eta)=\eta,\ \  F^{''}(\eta)=1.
\end{align}
%\end{lemma}

%
%In this section, we give the superconvergence analysis of PS \cite{P-S} element. We will use the mesh condition $\alpha$ and $RDP(N,\Psi)$ defined in subsection 2.3. Recall that for any $\mathbf u\in \mathbf V \bigcap H^3(\Omega)^2$, $\mathbf u^I\in \mathbf V_h$ denotes the isoparametric bilinear interpolation of $\mathbf u$ and for any $\mathbf\sigma\in\Sigma$, $\mathbf \sigma^I\in \Sigma_h$ is defined as equation \eqref{eq:interpolation}. We have the following lemmas.

\begin{lemma}\label{lem:gxiv_estimate}
Under \textbf{Diagonal condition} (MC1) and \textbf{Neighboring condition} (MC2),  for any $g\in H^2(\Omega)$ and $v\in S_h$ it hold
\begin{eqnarray}
\sum\limits_{K\in\mathcal T_h}h_K\int_{\hat K} g\xi\frac{\partial^2v}{\partial\xi\partial\eta} & \lesssim & h(h^{\alpha} |g|_1 + h|g|_2)|v|_1,\label{est1}\\
\sum\limits_{K\in\mathcal T_h}h_K\int_{\hat K} g\eta\frac{\partial^2v}{\partial\xi\partial\eta} & \lesssim & h(h^{\alpha} |g|_1 + h|g|_2)|v|_1.\label{est2}
\end{eqnarray}
\end{lemma}
\begin{proof}
We only give the proof of the first inequality, since the proof of the second one is similar. For any $K\in\mathcal T_h$,   $g\in H^2(\Omega)$ and $v\in S_h$, by \eqref{E-F}, integration by parts, Cauchy-Schwardz inequality and Lemma \ref{lem:dv_d_xi}, we have
\begin{eqnarray}\label{equality1}
h_K\int_{\hat K} \hat g\xi\frac{\partial^2\hat v}{\partial\xi\partial\eta} & = & -h_K\int_{\hat K}\frac{\partial \hat g}{\partial\xi}E(\xi)\frac{\partial^2\hat v}{\partial\xi\partial\eta}  \nonumber\\
&=&  h_K\int_{-1}^1 \left( \frac{\partial \hat g}{\partial\xi}E(\xi)\frac{\partial \hat v}{\partial\xi} \right)(\xi,-1)\dxi   -h_K \int_{-1}^1 \left( \frac{\partial \hat g}{\partial\xi}E(\xi)\frac{\partial \hat v}{\partial\xi} \right)(\xi,1)\dxi  \nonumber \\
& &\quad + h_K\int_{\hat K}\frac{\partial^2 \hat g}{\partial\xi\partial\eta}E(\xi)\frac{\partial \hat v}{\partial \xi} \nonumber\\
& = & \frac{h_K}{2}|l_l|\int_{l_l} E(\xi(s))\frac{\partial g}{\partial s}\frac{\partial v}{\partial s}\text{d} s - \frac{h_K}{2}|l_u|\int_{l_u} E(\xi(s))\frac{\partial g}{\partial s}\frac{\partial v}{\partial s}\text{d} s \nonumber\\
& & + \left( \mathcal O(h_K^{1+\alpha})  |g|_{1,K} +   \mathcal O(h_K^2)|g|_{2,K}    \right)|v|_{1,K},
\end{eqnarray}
where $l_u$ and $l_l$ are the upper and lower edges of $K$ (see Figure \ref{Fig:bilinear}). If the edge $l_u\subset \partial\Omega$, then the second term of the last equality in \eqref{equality1} vanishes due to the homogeneous Dirichlet boundary condition, i.e. $v|_{\partial \Omega}=0$.  If $l_u$ is an interior edge of the partition  $\mathcal T_h$,  we assume   $l_u$ is  shared by two elements, $K$ and $K_*$, of $ \mathcal T_h$.  By \textbf{Neighboring condition} (MC2) we have
$$
|h_{K} - h_{K_*}| = \mathcal O(h^{1+\alpha}),
$$
then,  from  trace inequality and inverse inequality, it follows
\begin{eqnarray}\label{h*}
& & |h_{K} - h_{K_*}||l_u|\int_{l_u}E(\xi(s)) \frac{\partial g}{\partial s}\frac{\partial v}{\partial s}\text{d} s\nonumber \\
%&\lesssim & h^{1+\alpha}|l_u|\left( h_K^{-1}\int_K|DgDv| +h_K \int_K|D^2gDv + DgD^2v| \right) \\
& \lesssim & h^{1+\alpha} (|g|_{1,K} + h_K|g|_{2,K})|v|_{1,K}.
\end{eqnarray}
The above arguments also apply to the edge $l_l$. As a result, a combination of \eqref{equality1}-\eqref{h*} yields the desired estimate \eqref{est1}.

 Similarly we can obtain \eqref{est2}.
\end{proof}

%The formulation of hybrid stress  FEM is: find
%$(\mathbf\sigma_{h},\textbf{u}_{h})\in \mathbf\sigma_{h}\times \mathbf V_{h}$, such that
%\begin{align}\label{FEM}
%\left\{\begin{array}{lll}
%a(\mathbf\sigma_{h},\mathbf\tau)-&b(\mathbf u_{h},\mathbf\tau)=0  &\quad\text{for all}\ \mathbf\tau\in \mathbf\sigma_{h};\\
%&b(\mathbf v,\mathbf\sigma_{h})=F(\mathbf v) &\quad\text{for all}\ \mathbf v\in \mathbf V_{h}.
%\end{array}
%\right.
%\end{align}

%Suppose 5-parameters constant element on $\hat{K}$ is
%\begin{align}
%\left(\begin{array}{lllll}
%1& 0& 0 & 0& 0 \\
%0& 1& 0 &0& 0\\
%0& 0& 1 & 0& 0
%\end{array}\right)\mathbf{\beta},\quad\mathbf{\beta}\in R^5,
%\end{align}

 \begin{lemma} \label{lem:super_a_b}
 Under \textbf{Diagonal condition} (MC1) and \textbf{Neighboring condition} (MC2), for  $\mathbf\sigma\in H^{2}(\Omega, \mathbb{R}_{sym}^{2\times 2})\cap \Sigma$ and $\mathbf u\in (H^{3}(\Omega))^2\cap \mathbf V$ it holds
 \begin{eqnarray}
\label{eq:a_super} a(\mathbf\sigma^{I}-\mathbf\sigma,\mathbf\tau) & = & 0,\qquad\qquad\text{for all }\mathbf \tau\in\Sigma_h,\\
\label{eq:b_sigma_super} b(\mathbf\sigma-\mathbf\sigma^{I},\mathbf v) & \lesssim & \left(h^{1+\alpha}||\mathbf\sigma||_{1}+h^2|\mathbf\sigma |_{2}\right)|\mathbf v|_{1},\quad \text{for all }\mathbf v\in \mathbf V_h,\\
\label{eq:b_u_super} b(\mathbf\tau,\mathbf u-\mathbf u^{I}) & \lesssim & h^{1+\alpha}||\mathbf u||_{3} \|\mathbf\tau\|,\quad\text{for all }\mathbf \tau\in \Sigma_h.
 \end{eqnarray}
 \end{lemma}
 \begin{proof}
The relation \eqref{eq:a_super} follows from  \eqref{eq:interpolation}, i.e.  the definition of $\sigma^I$.

Now we prove the estimate \eqref{eq:b_sigma_super}.  For any $\mathbf v\in \mathbf V_h$, we  decompose it as  $\mathbf v = \mathbf v_1 + \mathbf v_2$ with
$\mathbf {\hat v}_1=\mathbf v_1|_K \circ F_K\in\text{span}\{1,\ \xi,\ \eta\}^2$, $\mathbf {\hat v}_2=\mathbf v_2|_K\circ F_K\in\text{span}\{\xi\eta\}^2$, then it holds
\begin{eqnarray}\label{b()}
b(\mathbf v,\mathbf\sigma-\mathbf\sigma^{I}) & = & \sum\limits_{K\in\mathcal T_h}\int_{K}\mathbf\epsilon(\mathbf v)\cdot(\mathbf\sigma-\mathbf\sigma^{I}) \nonumber\\
& = & \sum\limits_{K\in\mathcal T_h}\int_{K}(\mathbf\sigma-\mathbf\sigma^{I})\cdot\tilde{\mathbf\epsilon}(\mathbf v) + \int_K(\mathbf\sigma-\mathbf\sigma^{I})\cdot\left(\mathbf\epsilon(\mathbf v) - \tilde{\mathbf\epsilon}(\mathbf v)\right) \nonumber\\
&= &\sum\limits_{K\in\mathcal T_h}\int_K(\mathbf\sigma-\mathbf\sigma^{I})\cdot \tilde{\mathbf\epsilon}(\mathbf v_1)  + \sum\limits_{K\in\mathcal T_h}\int_K(\mathbf\sigma^I - \mathbf\sigma)\cdot\left( \mathbf\epsilon(\mathbf v) - \tilde{\mathbf\epsilon}(\mathbf v) \right)\nonumber\\
& &+\sum\limits_{K\in\mathcal T_h}\left( \int_K\mathbf\sigma\cdot\tilde{\mathbf\epsilon}(\mathbf v_2)- \int_K\mathbf\sigma^{I}\cdot\tilde{\mathbf\epsilon}(\mathbf v_2)\right)\nonumber\\
&=: &  I_{1} + I_{2} + I_{3}.
\end{eqnarray}
We note that $J_K\tilde{\mathbf\epsilon}(\mathbf v_1)$ is a constant vector on $K$ by the definition \eqref{modified strain}. Thus, in view of Lemmas \ref{lem:dv_d_xi}-\ref{lem:L2app_sigmaI} we have
\begin{eqnarray}\label{I1}
|I_{1} |& = &|\sum\limits_{K\in\mathcal T_h} \int_{\hat K}(\hat{\mathbf\sigma} - \hat{\mathbf\sigma^I})\cdot J_K\tilde{\mathbf\epsilon}(\mathbf v_1) |= |\sum\limits_{K\in\mathcal T_h}J_K\tilde{\mathbf\epsilon}(\mathbf v_1)\cdot\int_{\hat K}(\hat{\mathbf\sigma} - \hat{\mathbf\sigma^I})|\nonumber\\
& \lesssim & \sum\limits_{K\in\mathcal T_h}h_K^{\alpha}\|\mathbf\sigma\|_{1,K}\|J_K\tilde{\mathbf\epsilon}(\mathbf v_1)\|_{0,\hat K} \nonumber\\
&  \lesssim &\sum\limits_{K\in\mathcal T_h}h_K^{1 + \alpha}\|\mathbf\sigma\|_{1,K}|\mathbf v_1|_{1,K}\nonumber\\
&\lesssim& h^{1 + \alpha}\|\mathbf\sigma\|_{1}|\mathbf v|_{1}.
\end{eqnarray}

For the term $I_{2}$, from Lemmas \ref{lem:L2app_sigmaI}-\ref{lem:epsilon_tilde} it follows
\begin{eqnarray}\label{I2}
I_{2} & \leq & \sum\limits_{K\in\mathcal T_h}\|\mathbf\sigma^I - \mathbf\sigma\|_{0,K}\|\mathbf\epsilon(\mathbf v) - \tilde{\mathbf\epsilon}(\mathbf v)\|_{0,K} \lesssim h^{1 + \alpha}\|\mathbf\sigma\|_{1}|\mathbf v|_{1}.
\end{eqnarray}

We turn to estimate $I_{3}$. Denote  $\hat{\mathbf v}_2 = \mathbf v_2|_K\circ F_K = :(u_0\xi\eta,v_0\xi\eta)^T$ and $\bs\hat{\sigma^I }= :A\mathbf\beta^I$. Then, by
\eqref{eq:sigma_I_K} and  \eqref{aa},
 we have
 $$
\mathbf\beta^I = H_K^{-1}\int_K A^T\mathbb C^{-1} \mathbf\sigma= \int_K\left(\frac{1}{4J_0}\left( \begin{array}{ccc} 1  &  0 & 0 \\ 0 & 1 & 0 \\ 0 & 0 & 1 \\ \\ d_{41}\eta     &    d_{42}\eta  & d_{43}\eta \\ \\
d_{51}\xi        &
d_{52}\xi       &
d_{53} \xi
\end{array}\right) +h.o.t.\right) \mathbf\sigma,
$$
which, together with   $u_0 = \frac{\partial^2 u}{\partial\xi\partial\eta },\  v_0 = \frac{\partial^2 v}{\partial\xi\partial\eta }$,   \textbf{Neighboring condition} (MC2), Lemma \ref{lem:gxiv_estimate}  and \eqref{hot}-\eqref{dij},  yields
%which, together with \eqref{aa}, yields
\begin{eqnarray}\label{I11}
|\sum\limits_{K\in\mathcal T_h}\int_K\sigma^I\cdot\tilde{\mathbf \epsilon}(\mathbf v_2)|& = & |\sum\limits_{K\in\mathcal T_h}\left(\mathbf\beta^I\right)^T \int_{\hat K} A^T \left(J_K\tilde{\mathbf\epsilon}(\mathbf v_2)\right)|\nonumber\\
& =&|\sum\limits_{K\in\mathcal T_h}  \frac{4J_0}{3}\left(\mathbf\beta^I \right)^T \left( \begin{array}{c} 0 \\ 0 \\ 0 \\ \frac{u_0+\frac{b_1}{a_1}v_0 }{a_1} \\  \frac{\frac{a_2}{b_2}u_0 + v_0}{b_2}  \end{array} \right)|\nonumber\\
% & = & \int_{\hat K} \bs\hat\sigma \cdot J_0\left( \begin{array}{c}
%\frac{1}{3a_1}(\frac{b_1}{a_1}v_0 + u_0)d_{41}\eta + \frac{1}{3b_2}(\frac{a_2}{b_2}u_0 + v_0)d_{51}\xi \\
%\frac{1}{3a_1}(\frac{b_1}{a_1}v_0 + u_0)d_{42}\eta + \frac{1}{3b_2}(\frac{a_2}{b_2}u_0 + v_0)d_{52}\xi \\
%\frac{1}{3a_1}(\frac{b_1}{a_1}v_0 + u_0)d_{43}\eta + \frac{1}{3b_2}(\frac{a_2}{b_2}u_0 + v_0)d_{53}\xi
%  \end{array} \right)+h.o.t.,\\
  &\lesssim& \left(h^{1+\alpha}|\mathbf\sigma|_{1}+h^2|\mathbf\sigma |_{2}\right)|\mathbf v|_{1}.
\end{eqnarray}
%where $|h.o.t|\lesssim h^{1+\alpha}$.
%\begin{equation}\label{hot}
%h.o.t=\frac{1}{J_0}(\tilde h_{ij}(\xi,\eta))_{5\times3}\text{ with } \max\limits_{i,j}\max\limits_{-1\leq\xi,\eta\geq1} |\tilde h_{ij}| \lesssim  h_K^\alpha.
%\end{equation}
%Noticing that   $u_0 = \frac{\partial^2 u}{\partial\xi\partial\eta },\  v_0 = \frac{\partial^2 v}{\partial\xi\partial\eta }$, by  \textbf{Neighboring condition} (MC2) and Lemma \ref{lem:gxiv_estimate}  we obtain
%\begin{equation}\label{I11}
%|\sum\limits_{K\in\mathcal T_h}\int_K\sigma^I\cdot\tilde{\mathbf \epsilon}(\mathbf v_2)|\lesssim \left(h^{1+\alpha}|\mathbf\sigma|_{1}+h^2|\mathbf\sigma |_{2}\right)|\mathbf v|_{1}.
%\end{equation}
Similarly, since
\begin{eqnarray*}
\int_K\mathbf\sigma\cdot\tilde{\mathbf\epsilon}(\mathbf v_2) & = & \int_{\hat K}\bs{\hat\sigma}\cdot \left( J_K\tilde{\mathbf\epsilon}(\hat{\mathbf v}_2) \right)  =
\int_{\hat K}\hat{\mathbf\sigma}\cdot \left( \begin{array}{c} u_0b_2\eta - u_0b_1\xi \\ -v_0a_2\eta + v_0a_1\xi \\ (v_0b_2 - u_0a_2)\eta + (u_0a_1 - v_0b_1)\xi \end{array}  \right),
\end{eqnarray*}
it follows
$$|\sum\limits_{K\in\mathcal T_h}\int_K\mathbf\sigma\cdot\tilde{\mathbf\epsilon}(\mathbf v_2)| \lesssim \left(h^{1+\alpha}|\mathbf\sigma|_{1}+h^2|\mathbf\sigma |_{2}\right)|\mathbf v|_{1},
$$
which, together with \eqref{I11}, yields
\begin{equation}\label{I3}
|I_3|\lesssim \left(h^{1+\alpha}|\mathbf\sigma|_{1}+h^2|\mathbf\sigma |_{2}\right)|\mathbf v|_{1}.
\end{equation}
As a result, the inequality \eqref{eq:b_sigma_super} follows from \eqref{b()}-\eqref{I2} and \eqref{I3}.

The thing left is to prove the estimate \eqref{eq:b_u_super}. Denote
$$
X_0 := \left( \begin{array}{cccc} b_2  & -b_1 & 0 & 0 \\ 0 & 0 & -a_2 & a_1 \\ -a_2 & a_1 & b_2 & b_1 \end{array} \right),\quad  X_1 := \left( \begin{array}{cccc}  b_{12}\xi  & -b_{12}\eta & 0 & 0 \\ 0 & 0 & -a_{12}\xi & a_{12}\eta \\ -a_{12}\xi & a_{12}\eta & b_{12}\xi & b_{12}\eta \end{array} \right),
$$
$$ \hat{\nabla}\hat{\mathbf u}:=\left(\frac{\partial \hat u}{\partial \xi},
 \frac{\partial \hat u}{\partial \eta},
 \frac{\partial \hat v}{\partial \xi},
 \frac{\partial \hat v}{\partial \eta}\right)^T \quad \text{ for } \mathbf u=(u,v),$$
and let $\mathbf u^b\in \mathbf V_h^b$ be   such that $\hat{\mathbf u}^I + \hat{\mathbf u}^b$ is the piecewise   quadratic  interpolation  of $\hat{\mathbf u}$ in the local coordinates $\xi,\eta$. Thanks to the relation \eqref{PS-relation} and the interpolation theory by \cite{Arnold-Boffi-Falk-2002} ,  for $\mathbf \tau\in \Sigma_h$ it holds
\begin{eqnarray*}
\int_{K}\mathbf\epsilon(\mathbf u-\mathbf u^{I})\cdot\mathbf\tau  & = & \int_{\hat{K}}\left( X_0\hat{\nabla}(\hat{\mathbf u} - \hat{\mathbf u}^I - \hat{\mathbf u}^b) + X_1\hat{\nabla}(\hat{\mathbf u} - \hat{\mathbf u^I} )\right)\cdot\hat{\mathbf\tau}\dxi \deta \\
& \lesssim & \left(h_K|\hat{\mathbf u}-\hat{\mathbf u}^{I}-\hat{\mathbf u}^{b}|_{1,\hat{K}} + h_K^{1 + \alpha}|\hat{\mathbf u} - \hat{\mathbf u}^{I}|_{1,\hat{K}}\right)\|\hat{\mathbf \tau}\|_{0,\hat{K}} \\
%& \lesssim & |\mathbf u - \mathbf u^{I} - \mathbf u^{b}|_{1,K}\|\mathbf \tau\|_{0,K} + h^{\alpha}|\mathbf u - \mathbf u^{I}|_{1,K}\|\mathbf \tau\|_{0,K} \\
& \lesssim & h^{1+\alpha}||\mathbf u||_{3,K}\|\mathbf \tau\|_{0,K}.
\end{eqnarray*}
%Similarly, by the fact that $\int_{K}\epsilon(\mathbf u^{b})\cdot\mathbf\tau =0,\ \text{for all }\mathbf\tau\in \Sigma_h^{EC}$, we have
%\begin{align*}
%\int_{K}\epsilon(\mathbf u-\mathbf u^{I})\cdot\mathbf\tau dxdy&=\int_{K}\epsilon(\mathbf u-\mathbf u^{I}-\mathbf u^{b})\cdot\mathbf\tau \lesssim h^2|\mathbf u|_{3,K}||\mathbf\tau||_{0,K}.
%\end{align*}
Then the desired inequality \eqref{eq:b_u_super} follows.
\end{proof}

We are now in a position to state the following superconvergence results for the hybrid stress method \eqref{eq:HybridFEM}.

\begin{theorem}\label{them:supper} Let $(\mathbf\sigma,\mathbf u)\in H^{2}(\Omega, \mathbb{R}_{sym}^{2\times 2})\bigcap \Sigma \times \mathbf V\bigcap (H^{3}(\Omega))^2$ and $(\mathbf\sigma_{h},\mathbf u_{h})\in \Sigma_h\times \mathbf V_h$ be the solutions of the problems \eqref{eq:model-weak} and \eqref{eq:HybridFEM}, respectively, and let $\mathbf u^{I}\in \mathbf V_{h}$ be the isoparametric bilinear interpolation of $\mathbf u$ and $\mathbf\sigma^{I}\in \Sigma_{h}$ be the projection of $\mathbf\sigma$  defined in \eqref{eq:interpolation}. Then, under  \textbf{Diagonal condition} (MC1) and \textbf{Neighboring condition} (MC2),
 it holds
\begin{eqnarray}
||\mathbf\sigma_{h}-\mathbf\sigma^{I}|| & \lesssim & h^{1+\alpha}||\mathbf\sigma||_{1}+h^2|\mathbf\sigma|_{2},\\
|\mathbf u_{h}-\mathbf u^{I}|_{1}  & \lesssim & h^{1+\alpha}(||\mathbf u||_{3}+||\mathbf\sigma||_{1})+h^2|\mathbf\sigma|_{2}.
\end{eqnarray}
\end{theorem}

\begin{proof} From \eqref{eq:model-weak} and \eqref{eq:HybridFEM} we easily obtain the error equations
%Because
%\begin{align}
%\left\{
%\begin{array}{llllll} \displaystyle a(\mathbf\sigma,\mathbf\tau)& + & b(\mathbf\tau,\mathbf u)& =& 0
%&\text{for all }\ \mathbf\tau\in \Sigma_{h}\subset\Sigma,\\
%\displaystyle & &  b(\mathbf\sigma,\mathbf v)& = & F(\mathbf v) &
%\text{for all }\ \mathbf v\in \mathbf V_{h}\subset\mathbf V,\end{array}\right.
%\end{align}
%and
%\begin{equation}
%\left\{ \begin{array}{llllll} a(\mathbf\sigma_h,\mathbf\tau) & + & b(\mathbf\tau,\mathbf u_h) & = & \mathbf 0 & \text{for all } \mathbf\tau\in\Sigma_h,\\
%& & b(\mathbf\sigma_h,\mathbf v) & = & F(\mathbf v) & \text{for all }\mathbf v\in\mathbf V_h,  \end{array}  \right.
%\end{equation}
\begin{eqnarray}
a(\mathbf\sigma-\mathbf\sigma_h,\mathbf\tau)  +  b(\mathbf\tau,\mathbf u-\mathbf u_h)&=&0\quad  \text{for all } \mathbf\tau\in\Sigma_h,\\
b(\mathbf\sigma-\mathbf\sigma_h,\mathbf v)  &=&0\quad  \text{for all } \mathbf v\in\mathbf V_h,
\end{eqnarray}
which, together with  the discrete inf-sup condition for $b(\cdot,\cdot)$ (cf. \cite{Yu-Xie-Carsten}), indicates
 \begin{align}|\mathbf u_{h}-\mathbf u^{I}|_{1}&\lesssim \sup_{\mathbf\tau\in \mathbf\Sigma_{h}}\frac{b(\mathbf\tau,\mathbf u_{h}-\mathbf u^{I})}{||\mathbf\tau||}=\sup_{\mathbf\tau \in \mathbf\Sigma_{h}}\frac{b(\mathbf\tau,\mathbf u_{h}-\mathbf u)+b(\mathbf\tau,\mathbf u-\mathbf u^{I})}{||\mathbf\tau||} \nonumber\\
 &=\sup_{\mathbf\tau\in \mathbf\Sigma_{h}}\frac{b(\mathbf\tau,\mathbf u-\mathbf u^{I})+a(\mathbf\sigma-\mathbf\sigma_{h},\mathbf\tau)}{||\mathbf\tau||}\nonumber\\
 &=\sup_{\mathbf\tau\in \mathbf\Sigma_{h}}\frac{b(\mathbf\tau,\mathbf u-\mathbf u^{I})+a(\mathbf\sigma-\mathbf\sigma^{I},\mathbf\tau)+a(\mathbf\sigma^{I}-\mathbf\sigma_{h},\mathbf\tau)}{||\mathbf\tau||}
 \end{align}
 and
\begin{align}
||\mathbf\sigma_{h}-\mathbf\sigma^{I}||\lesssim \sup_{\mathbf v\in \mathbf V_{h}}\frac{b(\mathbf\sigma_{h}-\mathbf\sigma^{I},\mathbf v)}{|\mathbf v|_{1}}=\sup_{\mathbf v\in \mathbf V_{h}}\frac{b(\mathbf\sigma-\mathbf\sigma^{I},\mathbf v)}{|\mathbf v|_{1}}.
\end{align}
Then the desired estimates follows from  the above two inequalities and Lemma \ref{lem:super_a_b}.
%, there hold
%\begin{eqnarray*}
%\|\mathbf\sigma_{h}-\mathbf\sigma^{I}\| & \lesssim & h^{1+\alpha}|\mathbf\sigma|_{1}+h^2|\mathbf\sigma|_{2}, \\
%|\mathbf u_{h}-\mathbf u^{I}|_{1} & \lesssim & h^{1+\alpha}|\mathbf u|_{2}+h^2|\mathbf u|_{3}+h^{1+\alpha}|\mathbf\sigma|_{1}+h^2|\mathbf\sigma|_{2}.
%\end{eqnarray*}
\end{proof}

\section{A posteriori error estimation of recovery type}

%In this section, we will discuss some recovery methods.
As shown in the estimate \eqref{a priori est}, the finite element solution $(\mathbf \sigma_h,\mathbf u_h)$ of PS hybrid stress method \eqref{eq:HybridFEM} is only of first order accuracy. We shall    show in Subsections \ref{sec5.1}-\ref{sec5.2}    that, by using the recovery techniques of \cite{NaGa-Zhang-2004,ZM-Z3,Shi-Xu-Zhang,Yan}, one can obtain recovered displacement gradients and  stresses of
improved accuracy, i.e.  $O(h^{1+\min\{\alpha,1\}})$.  Then, in Subsection \ref{sec5.3}, we  shall show the  asymptotical exactness of the  a posteriori error
estimators based on the recovered quantities.

\subsection{Gradient recovery by PPR}\label{sec5.1}
We follow the polynomial preserving recovery method (PPR) proposed in \cite{NaGa-Zhang-2004,ZM-Z3,Shi-Xu-Zhang} to construct the recovered displacement gradients
\begin{align}
G_{h}\mathbf u_{h}=(G_{h}u_{h}^{1},G_{h}u_{h}^{2})^{T}.
\end{align}
%
%now review  polynomial preserving recovery method (PPR), which were proposed in \cite{ZM-Z2}, on triangular meshes, then in \cite{ZM-Z3}, the author discussed the method for the second order elliptic problem on irregular quadrilaterals, later  Shi, Xu and Zhang generalized the results for elasticity problems \cite{Shi-Xu-Zhang}.
Here the gradient recovery operator $G_{h}:S_{h}\rightarrow S_{h}\times S_{h}$ is defined as follows \cite{ZM-Z3}: Given function $v_{h}\in S_h$, first define $G_{h}v_{h}$ at all nodes (vertices) of the partition $\mathcal{T}_{h}$, and then obtain $G_{h}v_{h}$ on the whole domain by interpolation using the original nodal shape functions of $S_{h}$.

In PPR  the values of $G_{h}v_{h}$ at all vertices of $\mathcal{T}_{h}$ are  determined   through the fitting method. In fact, let $Z_{i}(x_{i},y_{i})$ be any interior vertex of $\mathcal{T}_{h}$, and let $\omega_{i}$ be a   patch which consists of elements sharing the vertex $Z_{i}$, i.e.
%Gradient recovery by PPR seeks to find a function defined in an element patch $\omega_{i}$ surrounding the particular assembly node $Z_{i}(x_{i},y_{i})$. Usually, $\omega_{i}$ is choose to be the element patch that contains all elements adjacent to $Z_{i}$,
%that is
\begin{equation}\label{patch-Zi}
{\omega}_{i}:=\bigcup\{K\in\mathcal{T}_{h}: Z_{i} \text{ is a vertex of  }K\}.
\end{equation}
%To define $G_{h}v_{h} (Z_i)$ for
For convenience  all nodes on $\bar{\omega}_{i}$ (including $Z_{i}$) are denoted by  $Z_{ij}, j=1,2,\cdots, n (n\geqslant 6).$
%In the following, we will introduce the fitting method to recover the gradient of a function $v_{h}\in S_{h}$ at $Z_{i}$ by values of $Z_{ij}$.
We use local coordinates $(\hat x,\hat y)$ with $Z_{i}$ as the origin, i.e.
$(\hat{x},\hat{y})=\frac{(x,y)-(x_{i},y_{i})}{h}$, where  $h:=h_{i}$ denotes the length of the longest element edge in the patch $\omega_{i}$. The fitting polynomial is
\begin{eqnarray}
p_{2}(x,y;Z_{i})&=&
%\mathbf P^{T}\mathbf a=
\hat{\mathbf P}^{T}\hat{\mathbf  c}
%q_{2}(x,y;Z_{i})&=&\mathbf P^{T}_{2}\mathbf a_{2}= \hat{\mathbf P}^{T}_{2}\hat{\mathbf  a}_{2},
\end{eqnarray}
with
$$
%\hat{\mathbf P}&=(1,x,y,x^2,xy,y^2)^{T},\quad
\hat{\mathbf P}=(1,\hat{x},\hat{y},\hat{x}^2,\hat{x}\hat{y},\hat{y}^2)^{T},
%\mathbf a & =(a_{1},a_{2},a_{3},a_{4},a_{5},a_{6})^T,
\quad \hat{\mathbf c}=(c_{1},hc_{2},hc_{3},h^2c_{4},h^2c_{5},h^2c_{6})^T.
%\hat{\mathbf P}_{2}&=(1,x,y,x^2,xy,y^2,x^2 y,xy^2,x^2 y^2)^{T},\quad \hat{\mathbf P}_{2}=(1,\hat{x},\hat{y},\hat{x}^2,\hat{x}\hat{y},\hat{y}^2,\hat{x}^2\hat{y},\hat{x}\hat{y}^2,\hat{x}^2 \hat{y}^2)^{T}\\
%\mathbf a_{2} & =(a_{1},a_{2},a_{3},a_{4},a_{5},a_{6},a_{7},a_{8},a_{9})^T,\quad \hat{\mathbf a}_{2}=(a_{1},ha_{2},ha_{3},h^2a_{4},h^2a_{5},h^2a_{6},h^3a_{7},h^3a_{8},h^4a_{9})^T.
$$
The coefficient vector $\hat{\mathbf c}$ is determined by the linear system
\begin{align}\label{eq: a}
Q^{T}Q\hat{\mathbf c}=Q^{T}\mathbf b_{h},
\end{align}
where
$\mathbf{b}_{h}=(v_{h}(Z_{i1}),v_{h}(Z_{i2}),\cdots,v_{h}(Z_{in}))^T$ and
$$Q=
\left(
\begin{array}{llllll} \displaystyle
1 &\hat{x}_1 &\hat{y}_{1} &\hat{x}_{1}^2 &\hat{x}_{1}\hat{y}_{1} &\hat{y}_{1}^2\\
1 &\hat{x}_2 &\hat{y}_{2} &\hat{x}_{2}^2 &\hat{x}_{2}\hat{y}_{2} &\hat{y}_{2}^2\\
\vdots&\vdots&\vdots&\vdots&\vdots&\vdots\\
1 &\hat{x}_n &\hat{y}_{n} &\hat{x}_{n}^2 &\hat{x}_{n}\hat{y}_{n} &\hat{y}_{n}^2\\
\end{array}\right).
$$
%$$Q_{2}=
%\left(
%\begin{array}{lllllllll} \displaystyle
%1 &\hat{x}_1 &\hat{y}_{1} &\hat{x}_{1}^2 &\hat{x}_{1}\hat{y}_{1} &\hat{y}_{1}^2&\hat{x}_1^2 \hat{y}_1&\hat{x}_1\hat{y}_1^2&\hat{x}_1^2 \hat{y}_1^2\\
%1 &\hat{x}_2 &\hat{y}_{2} &\hat{x}_{2}^2 &\hat{x}_{2}\hat{y}_{2} &\hat{y}_{2}^2&\hat{x}_2^2 \hat{y}_2&\hat{x}_2\hat{y}_2^2&\hat{x}_2^2 \hat{y}_2^2\\
%\vdots&\vdots&\vdots&\vdots&\vdots&\vdots&\vdots&\vdots&\vdots\\
%1 &\hat{x}_n &\hat{y}_{n} &\hat{x}_{n}^2 &\hat{x}_{n}\hat{y}_{n} &\hat{y}_{n}^2&\hat{x}_n^2 \hat{y}_n&\hat{x}_n\hat{y}_n^2&\hat{x}_n^2 \hat{y}_n^2\\
%\end{array}\right).
%$$
Finally, define
\begin{align}
G_{h}v_{h}(Z_{i}):=\nabla p_{2}(0,0;Z_{i}).%, \quad\text{or}\quad
%G_{h}u_{h}(Z_{i})=\nabla q_{2}(0,0;Z_{i}).
\end{align}

As shown in \cite{ZM-Z3,Shi-Xu-Zhang},  under  \textbf{Diagonal condition} (MC1) and \textbf{Neighboring condition} (MC2), the gradient recovery operator $G_{h}$ is a bounded linear operator on the isoparametric bilinear displacement  finite element space $\mathbf V_{h}=S_h\times S_h$ in the followng sense:
\begin{equation}\label{bounded}
||G_{h}\mathbf v||\lesssim |\mathbf v|_{1},\quad\forall \mathbf v\in \mathbf V_{h}.
\end{equation}

In view of Theorem \ref{them:supper},  we can obtain the superconvergence  of the recovered displacement gradients $G_{h}\mathbf u_h$ by following the same routine as in the proof of Theorem 4.2 in \cite{ZM-Z3}.

%\begin{lemma} \label{them:G_h_stable}()
%Let $\mathcal{T}_h$ satisfy  \textbf{Diagonal condition} (MC1) and \textbf{Neighboring condition} (MC2). Then the gradient recovery operator $G_{h}$ is a bounded linear operator on the isoparametric bilinear displacement  finite element space $\mathbf V_{h}=S_h\times S_h$, i.e.
%\[||G_{h}\mathbf v||_{0}\lesssim |\mathbf v|_{1},\quad\forall \mathbf v\in \mathbf V_{h}.\]
%\end{lemma}
%\begin{theorem}
%Let $\mathcal{T}_h$ satisfy condition $\alpha$. Then the Recovery Operator $R_{h}$ is a bound linear operator on bilinear element space such that
%\[||R_{h}\mathbf \sigma||_{0}\lesssim ||\mathbf \sigma||_{0},\quad\forall \mathbf \sigma\in\mathbf \Sigma_{h}.\]
%\end{theorem}

\begin{theorem} \label{them:G_h_estimate}
Let $\mathbf u\in  \mathbf V\bigcap (H^{3}(\Omega))^2$ and $ \mathbf u_{h}\in  \mathbf V_h$ be the displacement solutions of the problems \eqref{eq:model-weak} and \eqref{eq:HybridFEM}, respectively. Under \textbf{Diagonal condition} (MC1) and \textbf{Neighboring condition} (MC2), the  gradient recovery is superconvergent in the sense that
\begin{align}
||\nabla \mathbf u-G_{h}\mathbf u_{h}||\lesssim h^{1+\alpha}(||\mathbf u||_{3}+\|\mathbf\sigma\|_{1})+h^2\|\mathbf\sigma\|_{2}.
\end{align}
\end{theorem}

\subsection{ Recovery of stresses}\label{sec5.2}

From the superconvergence of the recovered displacement gradients $G_{h}\mathbf u_{h}$ in Theorem \ref{them:G_h_estimate}, we can easily derive the following superconvergence of  the recovered stresses $G^{\bs\sigma}_{h}\bs\sigma_{h}=\frac{1}{2}\mathbb{C}\left(G_{h}\mathbf u_{h} + (G_{h}\mathbf u_{h})^T \right)$ for the stress tensor $\bs\sigma=\mathbb C \epsilon(\mathbf{u})$:
\begin{align}
||\bs\sigma-G^{\bs\sigma}_{h}\mathbf \sigma_{h}||\lesssim \|\mathbb C\| \left(h^{1+\alpha}(||\mathbf u||_{3}+\|\mathbf\sigma\|_{1})+h^2\|\mathbf\sigma\|_{2}\right).
\end{align}
However, due to the factor  $\|\mathbb C\|$ this estimate is not uniform with respect to the Lam\'e constant $\lambda$.

 In what follows we shall
    construct a uniform recovered-type stress approximation  by following   the idea of \cite{Yan}.

Denoting
\begin{align*}
M_{h}:=\{ v\in L^2(\Omega): \hat{v}=v|_{K}\circ F_{K} \in \text{ span}\{1,\xi,\eta,\xi\eta\} ,\ \text{for all}\ K\in \mathcal{T}_{h}\},
\end{align*}%$R_{h}:H^2(\Omega,\mathbb R_\text{sym}^{2\times 2}) \bigcap \Sigma\rightarrow S_{h}^3$
we  introduce a recovered-type operator $$R_h: L^2(\Omega)\rightarrow M_{h}$$ as follows. For any $\psi\in L^2(\Omega)$, we first define $R_{h}\psi$ at all vertices of $\mathcal{T}_{h}$,   then obtain $R_{h}\psi\in M_{h} $ on the whole domain by interpolation using the nodal shape functions of the piecewise isoparametric bilinear interpolation.

For any  interior vertex   $Z_{i}(x_{i},y_{i})$ of $\mathcal{T}_{h}$, we assume its patch $\omega_i$, defined in \eqref{patch-Zi}, consists of $N$ elements, $K_{1},K_{2},\cdots, K_{N}$, with $ N\geq3$.  To define $R_{h}\psi$ at  $Z_i$
%, with  the associated element patch  $ {\omega}_{0}=\bigcup\limits_{j=1}^ N K_j $,
we introduce the space  $W:=span\{1, x,y\}$ and let $\phi_i^\psi\in W$
%Gradient recovery by PPR seeks to find a function defined in an element patch $\omega_{i}$ surrounding the particular assembly node $Z_{i}(x_{i},y_{i})$. Usually, $\omega_{i}$ is choose to be the element patch that contains all elements adjacent to $Z_{i}$,
%that is
%$$\bar{\omega}_{i}:=\bigcup_{K\in\mathcal{T}_{h},Z_{i}\in \bar{K}}\bar{K}.$$
%Let $Z=(x_{0},y_{0})$ be a interior node, we select an element patch $\omega_{Z}$, where
%\[ \bar{\omega}_{Z}=\sum_{i=1}^{n}K_{i}, Z\in \overline{K}_{i}.\]
satisfy
\begin{align}\label{eq:stress_mim}
J(\phi_i^\psi)&=\min_{ w\in W}J(w),\quad J(w):=\sum_{j=1}^N\left(\int_{K_{j}}(w-\psi)\right)^2.
\end{align}
Assume $\phi^\psi=\alpha_1+\alpha_2x+\alpha_3y$ and denote
$$A_j:=\left( \int_{K_{j}} 1 ,\  \int_{K_{j}}x ,\ \int_{K_{j}}y\right), \quad A:=\left(A_1^T ,\  A_2^T ,\ \cdots,\ A_N^T\right)^T,$$
$$ \mathbf b:=\left( \int_{K_{1}}\psi,\   \int_{K_{2}}\psi,\ \cdots,\  \int_{K_{N}}\psi\right)^T,$$
then,  from \eqref{eq:stress_mim}, the constant vector $\alpha=(\alpha_1,\alpha_2,\alpha_3)^T$ is determined by
%$$\mathbf \phi^{\mathbf\tau}(x,y)= \left( \begin{array}{ccccccccc} 1  & x & y & 0 & 0 & 0 & 0 & 0 & 0   \\
%0 & 0 & 0 & 1 & x & y  & 0 & 0 & 0 \\  0& 0 & 0 & 0 & 0 & 0 & 1 & x & y  \end{array}  \right)\left(\begin{array}{l}
%\alpha_1\\ \vdots\\ \alpha_9 \end{array}\right)=:\Phi(x,y)\alpha,$$
%%\[ \mathbf w:=\left(\begin{array}{r}
%%d_{1}+d_{2}x+d_{3}y+d_{4}xy\\
%%d_{5}+d_{6}x+d_{7}y+d_{8}xy\\
%%d_{9}+d_{10}x+d_{11}y+d_{12}xy\\
%%\end{array}\right) \triangleq \Phi(x,y) \mathbf\alpha^{\mathbf w},
%%\]
%%where
%then from \eqref{eq:stress_mim} the constant vector $\alpha=(\alpha_1,\cdots,\alpha_9)^T$ is determined by
\begin{align}\label{eq:linear system}
A^{T}A\mathbf\alpha=A^{T}\mathbf b.
\end{align}
Thus it follows
\begin{equation}\label{phi}
 \phi_i^\psi=(1,x,y)\left(A^TA\right)^{-1}A^T\mathbf b.
 \end{equation}
%and
%\begin{equation}\label{phi}
%\mathbf \phi^{\mathbf\tau}=(\phi_{11},\phi_{22},\phi_{12})^T
%%\left((1,x,y)\left(A^TA\right)^{-1}A^T\left(\mathbf b_{11}, \  \mathbf b_{22},\
%% \mathbf b_{12}\right)\right)^T
% =\left(\mathbf b_{11}, \  \mathbf b_{22},\
% \mathbf b_{12}\right)^TA\left(A^TA\right)^{-1}(1,x,y)^T.
% \end{equation}
We hence define
\begin{equation}\label{Rh}
R_{h}\psi(Z_0)=\phi^\psi(Z_0).
\end{equation}

We next define $R_{h}\psi$ at  any  vertex   $Z_b\in \partial\Omega$. Let $Z_b$ be shared by $m$ ($m\geq 1$)   patches, e.g. $\omega_1,\omega_2,\cdots, \omega_m$, which are corresponding $m$ interior vertices $Z_1,Z_2,\cdots,Z_m$, then we can define
\begin{equation}\label{Rhb}
R_{h}\psi(Z_b):=\frac{1}{m}\sum\limits_{i=1}^m \phi_i^\psi(Z_b),
\end{equation}
 where $  \phi_i^\psi $ is given by \eqref{phi}.
% We choose $Z_0$ as follows.  If $Z_b$ is a vertex of only one element, then we choose the only interior vertex of the element as $Z_0$. If $Z_b$ is shared by more than one element, then  there exists an interior vertex $Z_0$ such that $Z_bZ_0$ is an edge of some element $K_b\in\mathcal{T}_{h}$,   we define

As a result,  for any given  stress finite element function $\mathbf\tau=(\tau_{11},\tau_{22},\tau_{12})^T\in \Sigma_h$, we  define the stress recovery
  $R_{h}\mathbf\tau\in M_{h}^3$ with
\begin{equation}\label{Rh-tau}
R_{h}\mathbf\tau:=(R_h\tau_{11},R_h\tau_{22},R_h\tau_{12})^T.
\end{equation}
%For $\mathbf\tau=(\tau_{11},\tau_{22},\tau_{12})^T\in \Sigma_h$, let $\mathbf \phi^{\mathbf\tau}=(\phi_{11},\phi_{22},\phi_{12})^T$ with $\phi_{lk}\in W:=span\{1, x,y\}$

%Following  a similar way to that of the gradient recovery in the above subsection, we  define the stress recovery operator
% $$R_{h}:
% % \Sigma_h
% L^2(\Omega)
%\rightarrow M_{h} $$
%
% $$R_{h}:
% % \Sigma_h
% \prod_{K\in\mathcal T_{h}}H^2(K,\mathbb R_\text{sym}^{2\times 2}) \bigcap \Sigma
%\rightarrow M_{h}^3$$ as follows.
%  We first define $R_{h}\mathbf\tau$ at all vertices, and then obtain $R_{h}\mathbf\tau\in M_{h}^3$ on the whole domain by interpolation using the nodal shape functions of the piecewise isoparametric bilinear interpolation.
%
 % the recovery can be determined from an element patch $\overline{\omega}_{Z_{i}}$. If $Z$ is covered by more than one element patches, then some averaging may be applied.

\begin{remark}\label{5.2}
 We can show that  $A^{T}A$ in \eqref{eq:linear system} is invertible for  sufficiently small $h$. Since $N\geq3$, it suffices to show $rank(A)=3$. In fact,
in view of  \eqref{eq:trans} and \eqref{lem:J0_J1_J2-2} it holds
\begin{eqnarray*}
A_j&=&\left( |K_{j}| ,\  a_0^{K_j}|K_j|+\frac43(a_1^{K_j}J_1^{K_j}+a_2^{K_j}J_2^{K_j}) ,\  b_0^{K_j}|K_j|+\frac43(b_1^{K_j}J_1^{K_j}+b_2^{K_j}J_2^{K_j})\right)\\
&=&\left( |K_{j}| ,\  a_0^{K_j}|K_j| ,\  b_0^{K_j}|K_j|\right)+O(h_{K_j}^{2+\alpha}),
\end{eqnarray*}
which implies
%$$A_i:=\frac{|K_{i}|}{4}\left(  1, \quad a_{0}^{K_{i}},\quad b_{0}^{K_{i}}\right)+h^{2+\alpha}\mathbf C,$$
\begin{equation}\label{A}
A=\left(\begin{array}{ccc}
   |K_{1}|&  a_0^{K_1}|K_1| &  b_0^{K_1}|K_1|\\
  |K_{2}|&  a_0^{K_2}|K_2| &  b_0^{K_2}|K_2|\\
  \vdots&\vdots&\vdots\\
    |K_{N}|&  a_0^{K_N}|K_N| &  b_0^{K_N}|K_N|
  \end{array}\right)+O(h^{2+\alpha}),
  \end{equation}
where   $|K_{j}|=O(h^2)$ is the area of  $K_{j}\subset w_{0}$. Recalling that  $Z_{0}$ is an interior vertex of $\mathcal{T}_{h}$  and $(a_0^{K_j},b_0^{K_j})$ is the center of the element $K_j$ ($1\leq j\leq N,  N\geq3$), we easily have the fact that there exist at least  three center points which are not lying on a same line. Thus, it holds $rank(A)=3$ for sufficiently small $h$.
\end{remark}
%When $h\rightarrow 0$,
%rank(A) is equal to rank(B)
%$$
%B=\left(\begin{array}{ccc}
%  1& a_{0}^{K_{1}} & b_{0}^{K_{1}} \\
%   1& a_{0}^{K_{2}} & b_{0}^{K_{2}} \\
%  \cdots&\cdots&\cdots\\
%   1& a_{0}^{K_{N}} & b_{0}^{K_{N}}
%  \end{array}\right).
%$$
%In fact, B has a full rank if and only if the centers of $K\in w_{0}$ are not lying on a same line,
%i.e. any interior node $Z_{0}$ is a common vertex of a least three quadrilaterals.
%So we can obtain that
%$rank(A)=3$ for $N\geq3$,
%which implies
%where
%$$A=(A_{1},A_{2},\cdots,A_{N})^{T},\quad
%B=(B_{1},B_{2},\cdots,B_{N})^{T},$$
%\begin{align*}
%A_{i}&=\left(\begin{array}{ccccccccc}
%\int_{K_{i}} & \int_{K_{i}}x &\int_{K_{i}}y&0&0&0& 0 & 0 & 0   \\
%0 & 0 & 0 & \int_{K_{i}}& \int_{K_{i}}x & \int_{K_{i}}y  & 0 & 0 & 0 \\  0& 0 & 0 & 0 & 0 & 0 & \int_{K_{i}} & \int_{K_{i}}x & \int_{K_{i}}y
%\end{array}
%\right),\\
%\mathbf B_{i}&=\left(\begin{array}{ccc}
%\int_{K_{i}}\tau_{11} & \int_{K_{i}}\tau_{22} &\int_{K_{i}}\tau_{12}
%\end{array}
%\right)^{T}.
%\end{align*}

%$$
%\Phi(x,y) = \left( \begin{array}{cccccccccccc} 1  & x & y & xy & 0 & 0 & 0 & 0 & 0 & 0 & 0 & 0   \\
%0 & 0 & 0 & 0 & 1 & x & y & xy & 0 & 0 & 0 & 0  \\ 0 & 0 & 0 & 0 & 0 & 0 & 0 & 0 & 1 & x & y & xy \end{array}  \right)
%$$
%and $\mathbf\alpha^{\mathbf w}\in\mathbb R^{12}$.

%When $Z\in \partial\Omega$, the recovery can be determined from an element patch $\overline{\omega}_{Z_{i}}$. If $Z$ is covered by more than one element patches, then some averaging may be applied.

By the definition of $R_h$, we can   derive Lemmas \ref{lemma:boun_R_h}-\ref{lem:R_h}.

\begin{lemma}\label{lemma:boun_R_h}
The operator $R_h: L^2(\Omega)\rightarrow M_{h}$ is bounded in $L^2$ norm with
%$$
%\|R_h\tau\| \lesssim \|\tau\|\qquad\text{for all }\tau\in L^2(\Omega)^3.
%$$
\begin{equation}\label{boun Rh}
\|R_h\psi\| \lesssim \|\psi\|,\quad \forall \psi\in L^2(\Omega).
\end{equation}
In addition, under  \textbf{Diagonal condition} (MC1) it holds
\begin{eqnarray}\label{est-Rh}
\|\psi - R_h\psi\|  \lesssim  h^{1 + \alpha}\|\psi\|_{1} + h^2 \|\psi\|_{2}, \quad\forall \psi\in H^2(\Omega).%,\forall K\in\mathcal T_h.
\end{eqnarray}

\end{lemma}
\begin{proof} We first prove \eqref{boun Rh}. Let $\mathcal V$ be the set of all vertices of $\mathcal T_h$.  For $\psi\in L^2(\Omega)$ and $Z_i  (x_i,y_i)\in\mathcal V$, let $\phi^\psi_i\in W$ be the solution of the minimization problem \eqref{eq:stress_mim}. From  \eqref{Rh}-\eqref{Rhb}
%Let $\varphi_i\in W$ be the solution of equation \eqref{eq:stress_mim} with $\tau_{kl} = \phi$ on the node $Z_i = (x_i,y_i)\in\mathcal V$, where $\mathcal V$ is the set of all node in $\mathcal T_h$. Then,
we have
\begin{eqnarray}\label{rh}
\|R_h\psi\|^2 & \thickapprox & h^2\sum\limits_{Z_i\in \mathcal V}\phi^\psi_i(Z_i)^2. %=  \sum\limits_{Z_i\in \mathcal V}\left(  (1,x_i,y_i)\left(A^TA\right)^{-1}A^T\mathbf b \right)^2
\end{eqnarray}
 Recalling that $||A||_\infty\lesssim h^2, ||(A^TA)^{-1}||_\infty\lesssim h^{-4}$ (cf. Remark \ref{5.2}) and $ \mathbf b=\left( \int_{K_{1}}\psi,\   \int_{K_{2}}\psi,\ \cdots,\  \int_{K_{N}}\psi\right)^T$,  from \eqref{phi}  we easily obtain
\begin{eqnarray*}
\phi_i^\psi(Z_i)^2 & = &  |(1,x_i,y_i)\left(A^TA\right)^{-1}A^T\mathbf b|^2\\
& \lesssim &h^{-2} \sum\limits_{j = 1}^{N}\|\psi\|_{0,K_j}^2,
\end{eqnarray*}
which, together with \eqref{rh}, leads to the desired conclusion.

  By noticing that the operator $R_h$ preserves linear polynomials on each patch $\omega_i$, namely $R_h\psi=\psi$ for $\psi\in W$, the desired estimate \eqref{est-Rh} follows from the Bramble-Hilbert lemma and \textbf{Diagonal condition} (MC1).
%For any interior node $Z_i = (x_i,y_i)$, assume that there are $N_i\geq 3$ elements $\{K_j \}_{j = 1}^{N_i}$ share $Z_i$. Let $\mathbf b_i = \left(\begin{array}{cccc} \int_{K_1}\phi & \int_{K_2}\phi &\cdots &\int_{K_{N_i}}\phi \end{array}\right)^T$. Then, we have
%
%In the first inequality, we have used the fact that the non-zero elements of $A$ are $\mathcal O(h^2)$, and then the non-zeros elements of $(A^TA)^{-1}$ are $\mathcal O(h^{-4})$.
%
%For the boundary nodes, we can get similar results. Then by finite overlapping, the desired result follows.
\end{proof}
\begin{lemma}\label{lem:R_h}
For $\mathbf\sigma\in \Sigma$,%H^2(\Omega,\mathbb R_\text{sym}^{2\times 2})\bigcap \Sigma$,
let   $\mathbf\sigma^{I}\in \Sigma_{h}$ be defined as in \eqref{eq:interpolation}. Then it holds
\begin{equation}\label{eq:R_h}
R_h\mathbf\sigma = R_h\mathbf\sigma^I.
\end{equation}
\end{lemma}
\begin{proof} In light of \eqref{Rh}-\eqref{Rhb}, it suffices to show $\phi^{\sigma_{il}}=\phi^{\sigma^I_{il}}$. By the relation  \eqref{lem:L2orth}  it holds $ \int_{K_j}\mathbf\sigma_{il}=\int_{K_j}\mathbf\sigma^I_{il}$ for $j=1,2,\cdots,N$. Then the   conclusion  follows from the minimization problem   \eqref{eq:stress_mim}.
\end{proof}

\begin{theorem}\label{5.74}
 Let $(\mathbf\sigma,\mathbf u)\in H^{2}(\Omega, \mathbb{R}_{sym}^{2\times 2})\bigcap \Sigma \times \mathbf V\bigcap (H^{3}(\Omega))^2$ and $(\mathbf\sigma_{h},\mathbf u_{h})\in \Sigma_h\times \mathbf V_h$ be the solutions of the problems \eqref{eq:model-weak} and \eqref{eq:HybridFEM}, respectively. Then, under  \textbf{Diagonal condition} (MC1) and \textbf{Neighboring condition} (MC2), the following superconvergent result holds:
\begin{align}\label{super-sig}
||\mathbf\sigma-R_{h}\mathbf \sigma_{h}||\lesssim h^{1+\alpha}\|\mathbf\sigma\|_{1}+h^2\|\mathbf\sigma\|_{2}.
\end{align}

\end{theorem}
\begin{proof} By Lemma \ref{lem:R_h} it holds
\begin{align*}%\label{eq:decompose}
\mathbf \sigma-R_{h}\mathbf \sigma_{h}=(\mathbf \sigma-R_{h}\mathbf \sigma)+R_{h}(\mathbf \sigma^{I}-\mathbf \sigma_{h}).
\end{align*}
Then %Note that $R_{h}\mathbf \sigma=R_{h}\mathbf \sigma^{I}$ since $\pi_{0}\mathbf\sigma$ is the integral average of $\mathbf\sigma$.
the desired superconvergence \eqref{super-sig} follows from Lemma \ref{lemma:boun_R_h} and Theorem
\ref{them:supper}.
\end{proof}

\subsection{A Posteriori Error Estimates}\label{sec5.3}

 Denote $e^{\mathbf u}:=||\nabla\mathbf u-\nabla\mathbf u_{h}||$, $e^{\mathbf \sigma}:=||\mathbf \sigma-\mathbf \sigma_{h}||$.  Recall that $G_{h}\mathbf u_{h}$ and $R_{h}\mathbf \sigma_{h}$ are the recovered displacement gradients and the recovered stresses, respectively. In what follows we shall use the a posteriori estimators
 $$\eta^{\mathbf u}=||G_{h}\mathbf u_{h}-\nabla \mathbf u_{h}||,\quad  \eta^{\mathbf\sigma}=||R_{h}\mathbf \sigma_{h}-\mathbf\sigma_{h}||$$
to  estimate the errors $ e^{\mathbf u}, e^{\mathbf \sigma}$.
% by computable quantity $\eta^{\mathbf u}$ , $\eta_{1}^{\mathbf\sigma}$, and $\eta_{2}^{\mathbf\sigma}$. According to Zienkiewicz-Zhu \cite{Z-Zhu2}, $\eta^{\mathbf u}=||G_{h}\mathbf u_{h}-\nabla \mathbf u_{h}||_{0}$ is the error estimator defined by the recovered gradient. $\eta_{1}^{\mathbf\sigma}=||G^{\mathbf \sigma}_{h}\mathbf \sigma_{h}-\mathbf\sigma_{h}||_{0}$ and $\eta_{2}^{\mathbf\sigma}=||R_{h}\mathbf \sigma_{h}-\mathbf\sigma_{h}||_{0}$ are the error estimator defined by the recovered stress tensor.

\begin{theorem}\label{them:post-u}
Assume that $\mathcal{T}_h$ satisfy \textbf{Diagonal condition} (MC1) and \textbf{Neighboring condition} (MC2).  Let $(\mathbf\sigma,\mathbf u)\in H^{2}(\Omega, \mathbb{R}_{sym}^{2\times 2})\bigcap \Sigma \times \mathbf V\bigcap (H^{3}(\Omega))^2$ and $(\mathbf\sigma_{h},\mathbf u_{h})\in \Sigma_h\times \mathbf V_h$ be the solutions of the problems \eqref{eq:model-weak} and \eqref{eq:HybridFEM}, respectively. Then it holds
\begin{align}
\eta^{\mathbf u}-\|\nabla \mathbf u-G_{h}\mathbf u_{h}\|&\leq e^{\mathbf u} \leq \eta^{\mathbf u}+\|\nabla \mathbf u-G_{h}\mathbf u_{h}\|, \label{a1}\\
%||G^{\mathbf \sigma}_{h}\mathbf \sigma_{h}-\mathbf\sigma_{h}||_{0}-\| \mathbf \sigma-G^{\mathbf \sigma}_{h}\mathbf \sigma_{h}\|_{0}&\leq ||\mathbf \sigma-\mathbf \sigma_{h}||_{0} \leq ||G^{\mathbf \sigma}_{h}\mathbf \sigma_{h}-\mathbf\sigma_{h}||_{0}+ \| \mathbf \sigma-G^{\mathbf \sigma}_{h}\mathbf \sigma_{h}\|_{0},\\
\eta^{\mathbf \sigma}-\| \mathbf \sigma-R_{h}\mathbf \sigma_{h}\|&\leq e^{\mathbf \sigma} \leq\eta^{\mathbf \sigma}+ \| \mathbf \sigma-R_{h}\mathbf \sigma_{h}\|.\label{a2}
\end{align}
Moreover, if   the solution  $(\mathbf\sigma_{h},\mathbf u_{h}) $ is such that $ ||\nabla\mathbf u-\nabla\mathbf u_{h}||\gtrsim h$ and $||\mathbf \sigma-\mathbf \sigma_{h}|| \gtrsim h$, then the recovery type a posterior error estimators $ \eta^{\mathbf u}, \eta^{\mathbf \sigma}$   are asymptotically exact in the sense
\begin{align}\label{a3}
 { \eta^{\mathbf u}}/{ e^{\mathbf u}}=1+O(h^{\min\{\alpha,1\}}),\quad
%\frac{||G^{\mathbf \sigma}_{h}\mathbf \sigma_{h}-\mathbf\sigma_{h}||_{0}}{ ||\mathbf \sigma-\mathbf \sigma_{h}||_{0} }&=1+\|\mathbb C\|O(h^{\rho}),\quad \rho=\min(1,\alpha),\\
 {\eta^{\mathbf \sigma}}/{ e^{\mathbf \sigma}}=1+O(h^{\min\{\alpha,1\}}).
\end{align}
\begin{proof} The inequalities \eqref{a1}-\eqref{a2} follow from  triangular inequality directly, and  the estimates \eqref{a3} follow from \eqref{a1}-\eqref{a2}, Theorem \ref{them:G_h_estimate} and Theorem \ref{5.74}.

\end{proof}
\end{theorem}

\section{Numerical Experiments}
In
this section we compute two test problems, Examples \ref{ex:strain_1}-\ref{ex:strain_2}, to verify our  results of superconvergence and a posterior error estimation   for  the PS hybrid stress finite element method.
  The    examples are  both plane strain problems with pure displacement boundary conditions, where  the Lam\'{e} parameters  $\mu,\ \lambda$ are given
by $$\displaystyle \mu=\frac{E}{2(1+\nu)},\quad \lambda=\frac{E\nu}{(1+\nu)(1-2\nu)},$$ with $0<\nu<0.5$ the Poisson ratio and $E$ the Young's modulus. We set $E=1500$.
In all the computation we use $4\times 4$
Gaussian quadrature. Notice that $2\times 2$ Gaussian quadrature is
accurate for computing the stiffness matrix of the PS hybrid stress FEM.  All the fine meshes are obtained by bisection scheme.
We compute  the following relative errors for the displacement and stress approximation:
$$\bar{\theta}^{\mathbf u}: = \frac{|\mathbf u_{h}-\mathbf u^{I}|_{1}}{|\mathbf u|_1}, \quad \bar{e}^{\mathbf u}: = \frac{|\mathbf u_h - \mathbf u|_1}{|\mathbf u|_1}, \quad \bar{\eta}^{\mathbf u}:=\frac{{\eta}^{\mathbf u}}{|\mathbf u|_1}=\frac{||G_{h}\mathbf u_h - \nabla\mathbf u_{h}||}{|\mathbf u|_1},$$
 $$\bar{\theta}^{\mathbf\sigma}:=\frac{\|\mathbf\sigma_{h}-\mathbf\sigma^{I}\|}{\|\mathbf\sigma\|}, \quad \bar{e}^{\mathbf\sigma}: = \frac{\|\mathbf\sigma - \mathbf\sigma_h\|}{\|\mathbf\sigma\|},\quad \bar{\eta}^{\mathbf\sigma}:=\frac{{\eta}^{\mathbf\sigma}}{||\mathbf\sigma||}=\frac{||R_{h}\mathbf \sigma_{h}-\mathbf\sigma_{h}||}{||\mathbf\sigma||}.$$

\begin{example}\label{ex:strain_1} %(Plane strain test 1)
The domain $\Omega=[0,1]\times [0,1]$, %A plane strain pure bending body is used to test locking-free performance, with domain and
 the body force
$$
\mathbf f = E\pi^2\left( \begin{array}{c} \cos(\pi x)\sin(\pi y) \\ -\sin(\pi x)\cos(\pi y) \end{array} \right),
$$
%The surface traction $\mathbf g $ on $\Gamma_N = \{ (x,y)\in [0,1]\times[0,1]:\ x = 1 \text{ or } y = 0\text{ or } 1\}$ is given by $\mathbf g|_{x = 1} = (-2Ey,0)^T$, and $\mathbf g|_{y = 0 \text{ or }1} = (0,0)^T$, and
 and the exact solution $(\mathbf{u,\sigma})$ is given by
$$
\mathbf u = \left(\begin{array}{c} u \\ v \end{array}\right) = \left( \begin{array}{c} (1 + \nu)\cos(\pi x)\sin(\pi y) - 2(1-\nu^2)xy \\ \\
-(1+\nu)\sin(\pi x)\cos(\pi y) + (1-\nu^2)x^2 + \nu(1+\nu)(y^2 - 1) \end{array}\right),
$$
$$
\mathbf\sigma = E\left( \begin{array}{cc} -\pi\sin(\pi x)\sin(\pi y) -2y & 0\\ \\
0 & \pi\sin(\pi x)\sin(\pi y) \end{array}\right).
$$
The initial mesh is shown   in Figure \ref{Fig:Initial mesh}, and  numerical results are listed in Table \ref{table:strain_1_irr_PS}.
\end{example}

\begin{example}\label{ex:strain_2} %(Plane strain test 2)
The domain $\Omega=[0,10]\times[-1,1]$,
%A plane strain pure bending cantilever beam is used to test locking-free performance, with domain and meshes as in Figs. \ref{fig:beam_Len} and \ref{Fig:reandirre}.
%In this case,
 the body force
$$
\mathbf f = E\pi^2\left( \begin{array}{c} \cos(\pi x)\sin(\pi y) \\ -\sin(\pi x)\cos(\pi y) \end{array} \right),
$$
%The surface traction $\mathbf g $ on $\Gamma_N = \{ (x,y)\in [0,10]\times[-1,1]:\ x = 10 \text{ or } y = \pm 1\}$ is given by $\mathbf g|_{x = 10} = (-2Ey,0)^T$, and $\mathbf g|_{y = \pm 1} = (0,0)^T$,
and the exact solution is given by
$$
\mathbf u = \left(\begin{array}{c} u \\ v \end{array}\right) = \left( \begin{array}{c} (1 + \nu)\cos(\pi x)\sin(\pi y) - 2(1-\nu^2)xy \\ \\
-(1+\nu)\sin(\pi x)\cos(\pi y) + (1-\nu^2)x^2 + \nu(1+\nu)(y^2 - 1) \end{array}\right),
$$
$$
\mathbf\sigma = E\left( \begin{array}{cc} -\pi\sin(\pi x)\sin(\pi y) -2y & 0\\ \\
0 & \pi\sin(\pi x)\sin(\pi y) \end{array}\right).
$$
The initial mesh is shown   in Figure \ref{fig:beam_Len},
and  numerical results are listed in Table \ref{table:strain_2_irr_PS}.% \ref{table:strain_1_reg}-
\end{example}

We note that the refinement by bisection means that \textbf{Diagonal condition} (MC1) is satisfied with $\alpha=1$.  From Tables \ref{table:strain_1_irr_PS}-\ref{table:strain_2_irr_PS} we can draw the following conclusions.

\begin{itemize}

\item  $\bar{\theta}^{\mathbf u}$ and  $\bar{\theta}^{\mathbf \sigma}$ are of  second order convergence, uniformly with respect to $\lambda$. These are   conformable to the uniform superconvergence results in Theorem \ref{them:supper}.

\item   $\bar{e}^{\mathbf u}$  and $\bar{\eta}^{\mathbf u}$, as well as   $\bar{e}^{\mathbf \sigma}$ and $\bar{\eta}^{\mathbf \sigma}$, are of first order convergence, uniformly with respect to $\lambda$. In particular,    $\bar{\eta}^{\mathbf u}$ and $\bar{\eta}^{\mathbf \sigma}$ are asymptotically exact, which means the a posteriori estimators $ {\eta}^{\mathbf u}$ and $ {\eta}^{\mathbf \sigma}$ are asymptotically exact.  All these are conformable to the a posterior estimates in Theorem \ref{them:post-u}.

\end{itemize}

\begin{figure}[htdp]
\setlength{\unitlength}{1cm}
\begin{picture}(14,7)
\put(4,1){\line(1,0){4}}\put(4,1){\line(0,1){4}}\put(4,5){\line(1,0){4}}\put(8,5){\line(0,-1){4}}
%%connect midpoint
\put(4,2.2){\line(5,2){2}}\put(6,3){\line(5,1){2}}
\put(6,3){\line(-2,5){0.8}}\put(6,3){\line(-1,-5){0.4}}

\put(4,0.6){$ (0,0)$}
\put(5.6,0.6){$(0.4,0)$}
\put(8,0.6){$ (1,0)$}
\put(2.8,2.2){$ (0,0.3) $}
\put(6.1,2.7){$ (0.5,0.5) $}
\put(8.1,3.4){$(1,0.6)$}
\put(4,5.1){$ (0,1)$}
\put(5.2,5.1){$ (0.3,1)$}
\put(8,5.1){$ (1,1)$}
\end{picture}
\caption{ $2\times 2$ irregular mesh for Example \ref{ex:strain_1}}\label{Fig:Initial mesh}
\end{figure}

\begin{table}[h]
\caption{The results of PS element on irregular meshes: Example \ref{ex:strain_1}.}\label{table:strain_1_irr_PS}
\begin{center}
\begin{tabular}{|c|c|c|c|c|c|c|c|c|}
\hline
$\nu$  & Error     & $8\times 8$     &  $16\times 16$  &  $32\times 32$   &   $64\times 64$   &  $128\times 128$   &  Order\\
\hline
0.3     & $\bar \theta^{\mathbf u}$                       & 0.0051  &  0.0013  &  0.0003  &  0.0001 &   0.0000 & 1.98\\
        &   $\bar e^{\mathbf u}$                          & 0.1114  &  0.0556  &  0.0278  &  0.0139 &   0.0069 & 1.00\\
        &   $\bar \eta^{\mathbf u}$                       & 0.1216  &  0.0573  &  0.0280  &  0.0139 &   0.0069 & 1.03\\
        \cline{2-8}
        &   $\bar \theta^{\mathbf\sigma}$                 & 0.0138  &  0.0034  &  0.0009  &  0.0002 &   0.0001 & 2.00\\
        &   $\bar e^{\mathbf\sigma}$                      & 0.0953  &  0.0475  &  0.0237  &  0.0119 &   0.0059 & 1.00\\
        %&   $\bar \eta_{1}^{\mathbf\sigma}$               & 0.1116 &   0.0498  &  0.0240  &  0.0119  &  0.0059 & 1.06\\
        &   $\bar \eta^{\mathbf\sigma}$                   & 0.1059  &  0.0491  &  0.0240  &  0.0119 &   0.0059 & 1.04\\
\hline
0.49    & $\bar \theta^{\mathbf u}$                       & 0.0054  &  0.0014  &  0.0004  &  0.0001 &   0.0000 & 1.98\\
        &   $\bar e^{\mathbf u}$                          & 0.1143  &  0.0569  &  0.0284  &  0.0142 &   0.0071 & 1.00 \\
        &   $\bar \eta^{\mathbf u}$                       & 0.1240  &  0.0586  &  0.0287  &  0.0142 &   0.0071 & 1.03\\
        \cline{2-8}
        &   $\bar \theta^{\mathbf\sigma}$                 & 0.0153  &  0.0038  &  0.0010  &  0.0002 &   0.0001 & 2.00\\
        &   $\bar e^{\mathbf\sigma}$                      & 0.1182  &  0.0593  &  0.0297  &  0.0148 &   0.0074 & 1.00\\
        %&   $\bar \eta_{1}^{\mathbf\sigma}$              & 0.2250 &   0.0677  &  0.0308  &  0.0150  &  0.0074 & 1.23\\
        &   $\bar \eta^{\mathbf\sigma}$                   & 0.1294  &  0.0609  &  0.0299  &  0.0149 &   0.0074 & 1.03\\
\hline
%0.499    &  $\bar \theta^{\mathbf u}$                     & 0.0194 &   0.0051  &  0.0013  &  0.0003  &  0.0001 & 1.97\\
%        &   $\bar e^{\mathbf u}$                          & 0.1159 &   0.0572  &  0.0285  &  0.0142  &  0.0071 & 1.01\\
%        &   $\bar \eta^{\mathbf u}$                       & 0.1222 &   0.0585  &  0.0287  &  0.0143  &  0.0071 & 1.03\\
%        &   $\bar \theta^{\mathbf\sigma}$                 & 0.0210 &   0.0053  &  0.0013  &  0.0003  &  0.0001 & 2.00\\
%        &   $\bar e^{\mathbf\sigma}$                      & 0.1191 &   0.0601  &  0.0301  &  0.0151  &  0.0075 & 1.00\\
%        &   $\bar \eta_{1}^{\mathbf\sigma}$               & 1.9913 &   0.3412  &  0.0861  &  0.0257  &  0.0092 & 1.94\\
%        &   $\bar \eta_{2}^{\mathbf\sigma}$               & 0.1396 &   0.0663  &  0.0326  &  0.0162  &  0.0081 & 1.03\\
%\hline
0.4999  &  $\bar \theta^{\mathbf u}$                      & 0.0057  &  0.0014  &  0.0004  &  0.0001 &   0.0000 & 1.99\\
        &   $\bar e^{\mathbf u}$                          & 0.1144  &  0.0570  &  0.0285  &  0.0142 &   0.0071 & 1.00\\
        &   $\bar \eta^{\mathbf u}$                       & 0.1241  &  0.0587  &  0.0287  &  0.0143 &   0.0071 & 1.03\\
        \cline{2-8}
        &   $\bar \theta^{\mathbf\sigma}$                 & 0.0155  &  0.0039  &  0.0010  &  0.0002 &   0.0001 & 2.00\\
        &   $\bar e^{\mathbf\sigma}$                      & 0.1203  &  0.0604  &  0.0302  &  0.0151 &   0.0076 & 1.00\\
        %&   $\bar \eta_{1}^{\mathbf\sigma}$               & 20.0477&   3.4066  &  0.8159  &  0.2097  &  0.0537 & 2.14\\
        &   $\bar \eta^{\mathbf\sigma}$                   & 0.1315  &  0.0620  &  0.0304  &  0.0151 &   0.0076 & 1.03\\
\hline
\end{tabular}
\end{center}
\end{table}

\begin{figure}[h]
\setlength{\unitlength}{0.8cm}
\begin{picture}(14,4)

\put(2,1){\line(1,0){10}} \put(2,1){\line(0,1){2}} \put(2,3){\line(1,0){10}} \put(12,1){\line(0,1){2}}

\put(1.7,1){\line(0,1){2}} \put(1.3,1){\line(1,0){0.4}} \put(1.3,3){\line(1,0){0.4}} \put(1.4,1.9){$2$}

\put(2,0.7){\line(1,0){10}} \put(2,0.3){\line(0,1){0.4}}  \put(12,0.3){\line(0,1){0.4}} \put(3,0.3){\line(0,1){0.4}}
\put(4,0.3){\line(0,1){0.4}} \put(6,0.3){\line(0,1){0.4}} \put(9,0.3){\line(0,1){0.4}}
\put(2.4, 0.3){$1$} \put(3.4, 0.3){$1$} \put(4.9, 0.3){$2$} \put(7.4, 0.3){$3$} \put(10.4, 0.3){$3$}

\put(2,3.3){\line(1,0){10}} \put(2,3.3){\line(0,1){0.4}} \put(12,3.3){\line(0,1){0.4}}
\put(4,3.3){\line(0,1){0.4}} \put(6,3.3){\line(0,1){0.4}} \put(7,3.3){\line(0,1){0.4}} \put(8,3.3){\line(0,1){0.4}}
\put(2.9,3.5){$2$} \put(4.9,3.5){$2$} \put(6.4,3.5){$1$} \put(7.4,3.5){$1$} \put(9.9,3.5){$4$}

\put(1.9,1.9){$\bullet$} \put(2.1,1.9){$(0,0)$}

\put(3,1){\line(1,2){1}} \put(4,1){\line(1,1){2}} \put(6,1){\line(1,2){1}} \put(9,1){\line(-1,2){1}}

\end{picture}
\caption{  $5\times 1$ irregular mesh for Example \ref{ex:strain_2}.}\label{fig:beam_Len}
\end{figure}

\begin{figure}[htdp]
\setlength{\unitlength}{1cm}
\begin{picture}(12,2.0)
%rectangular 1
%\put(0.5,3){\line(1,0){5}}\put(0.5,3){\line(0,1){1}}\put(0.5,4){\line(1,0){5}}\put(5.5,3){\line(0,1){1}}

%connect
%\put(1.5,3){\line(0,1){1}}\put(2.5,3){\line(0,1){1}}\put(3.5,3){\line(0,1){1}}\put(4.5,3){\line(0,1){1}}

%label
%\put(2.5,2.5){\footnotesize $5\times1$}

%rectangular 2
\put(0.5,1){\line(1,0){5}}\put(0.5,1){\line(0,1){1}}\put(0.5,2){\line(1,0){5}}\put(5.5,1){\line(0,1){1}}

%connect
\put(1,1){\line(1,2){0.5}}\put(1.5,1){\line(1,1){1}}\put(2.5,1){\line(1,2){0.5}}\put(4,1){\line(-1,2){0.5}}

%label
\put(3,0.5){\footnotesize $5\times1$}

%rectangular 3
%\put(0.5,1){\line(1,0){5}}\put(0.5,1){\line(0,1){1}}\put(0.5,2){\line(1,0){5}}\put(5.5,1){\line(0,1){1}}
%
%%connect
%\put(1.5,1){\line(0,1){1}}\put(2.5,1){\line(0,1){1}}\put(3.5,1){\line(0,1){1}}\put(4.5,1){\line(0,1){1}}
%\put(0.5,1.5){\line(1,0){5}}\put(1,1){\line(0,1){1}}\put(2,1){\line(0,1){1}}\put(3,1){\line(0,1){1}}\put(4,1){\line(0,1){1}}
%\put(5,1){\line(0,1){1}}
%
%%label
%\put(2.5,0.5){\footnotesize $10\times2$}

%rectangular 4
\put(6.5,1){\line(1,0){5}}\put(6.5,1){\line(0,1){1}}\put(6.5,2){\line(1,0){5}}\put(11.5,1){\line(0,1){1}}

%connect
\put(7,1){\line(1,2){0.5}}\put(7.5,1){\line(1,1){1}}\put(8.5,1){\line(1,2){0.5}}\put(10,1){\line(-1,2){0.5}}
\put(6.5,1.5){\line(1,0){5}}\put(6.75,1){\line(1,4){0.25}}\put(7.25,1){\line(3,4){0.75}}\put(8,1){\line(3,4){0.75}}
\put(9.25,1){\line(0,1){1}}\put(10.75,1){\line(-1,4){0.25}}

%label
\put(9,0.5){\footnotesize $10\times2$}
\end{picture}
\caption{Irregular meshes}\label{Fig:reandirre}
\end{figure}
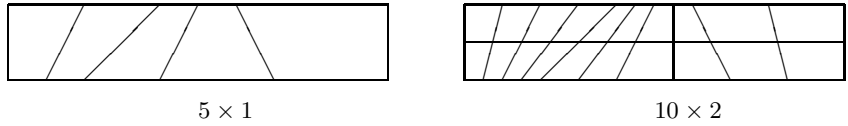

\begin{table}[h]
\caption{The results of PS element on irregular meshes: Example \ref{ex:strain_2}.}\label{table:strain_2_irr_PS}
\begin{center}
\begin{tabular}{|c|c|c|c|c|c|c|c|c|}
\hline
$\nu$  & Error        &  $20\times 4$  &  $40\times 8$   &   $80\times 16$   &  $160\times 32$    & $320\times 64$ & Order \\
\hline
0.3     &  $\bar\theta^{\mathbf u}$                       & 0.0478  &  0.0132  &  0.0037  &  0.0009  &  0.0002  & 1.92\\
        &   $\bar e^{\mathbf u}$                          & 0.1712  &  0.0889  &  0.0448  &  0.0224  &  0.0112  & 0.98\\
        &   $\bar \eta^{\mathbf u}$                       & 0.1760  &  0.1267  &  0.0607  &  0.0287  &  0.0140  & 0.91\\
        \cline{2-8}
        &   $\bar \theta^{\mathbf\sigma}$                 & 0.2848  &  0.0692  &  0.0168  &  0.0042  &  0.0010  & 2.03\\
        &   $\bar e^{\mathbf\sigma}$                      & 0.5451  &  0.2741  &  0.1362  &  0.0680  &  0.0340  & 1.00\\
        %&   $\bar \eta_{1}^{\mathbf\sigma}$               & 0.5789&    0.4062&    0.1723&    0.0740&    0.0349  & 1.01\\
        &   $\bar \eta^{\mathbf\sigma}$                   & 0.4972  &  0.3447  &  0.1548  &  0.0710  &  0.0345  & 0.96\\
\hline
0.49    &  $\bar \theta^{\mathbf u}$                      &  0.1203 &   0.0268 &   0.0065 &   0.0016 &   0.0004  & 2.05\\
        &   $\bar e^{\mathbf u}$                          &  0.2501 &   0.1204 &   0.0596 &   0.0297 &   0.0149  & 1.02 \\
        &   $\bar \eta^{\mathbf u}$                       &  0.2006 &   0.1551 &   0.0755 &   0.0355 &   0.0173  & 0.89\\
        \cline{2-8}
        &   $\bar \theta^{\mathbf\sigma}$                 & 0.4431  &  0.1060  &  0.0254  &  0.0063  &  0.0016  & 2.04 \\
        &   $\bar e^{\mathbf\sigma}$                      & 0.6416  &  0.3286  &  0.1635  &  0.0816  &  0.0408  & 1.00\\
        %&   $\bar \eta_{1}^{\mathbf\sigma}$               & 2.5132&    1.7287&    0.6041&    0.1787&    0.0577  & 1.36 \\
        &   $\bar \eta^{\mathbf\sigma}$                   & 0.5682  &  0.3974  &  0.1816  &  0.0845  &  0.0412  & 0.95\\
\hline
%0.499   &   $\bar \theta^{\mathbf u}$                     & 0.1165&    0.0350&    0.0090&    0.0023&    0.0006  & 1.91\\
%        &   $\bar e^{\mathbf u}$                          & 0.2397&    0.1223&    0.0606&    0.0302&    0.0151  & 1.00 \\
%        &   $\bar \eta^{\mathbf u}$                       & 0.1965&    0.1537&    0.0758&    0.0358&    0.0175  & 0.87\\
%        &   $\bar \theta^{\mathbf\sigma}$                 & 0.4786&    0.1440&    0.0359&    0.0089&    0.0022  & 1.94\\
%        &   $\bar e^{\mathbf\sigma}$                      & 0.6469&    0.3380&    0.1665&    0.0828&    0.0414  & 0.99\\
%        &   $\bar \eta_{1}^{\mathbf\sigma}$               & 24.3260&   16.6103&   5.6386&    1.5358&    0.3947  & 1.49\\
%        &   $\bar \eta_{2}^{\mathbf\sigma}$               & 0.6212&    0.4091&    0.1856&    0.0859&    0.0419  & 0.97\\
%\hline
0.4999    & $\bar \theta^{\mathbf u}$                     & 0.1450  &  0.0289  &  0.0070  &  0.0017  &  0.0004  & 2.09\\
        &   $\bar e^{\mathbf u}$                          & 0.2653  &  0.1229  &  0.0607  &  0.0303  &  0.0151  & 1.03\\
        &   $\bar \eta^{\mathbf u}$                       & 0.2030  &  0.1571  &  0.0765  &  0.0359  &  0.0175  & 0.89\\
        \cline{2-8}
        &   $\bar \theta^{\mathbf\sigma}$                 & 0.4981  &  0.1101  &  0.0263  &  0.0065  &  0.0016  & 2.07\\
        &   $\bar e^{\mathbf\sigma}$                      & 0.6702  &  0.3341  &  0.1661  &  0.0829  &  0.0414  & 1.00\\
        %&   $\bar \eta_{1}^{\mathbf\sigma}$               & 243.1669&  165.9394&  56.2687&   15.3029&   3.9162  &1.49\\
        &   $\bar \eta^{\mathbf\sigma}$                   & 0.5732  &  0.4024  &  0.1841  &  0.0858  &  0.0419  & 0.95\\
\hline
\end{tabular}
\end{center}
\end{table}

\end{document}